\documentclass[a4paper,oneside]{article}
\usepackage[top=4cm,bottom=3.5cm,left=3.5cm,right=3.5cm]{geometry}
\usepackage[T1]{fontenc}
\usepackage[utf8]{inputenc}
\usepackage[UKenglish]{babel}
\usepackage[usenames,dvipsnames]{xcolor}
\usepackage{amsthm}
\usepackage{amsmath, amssymb}
\usepackage{bbm}
\usepackage{nicefrac}
\usepackage{mathtools}
\usepackage{mathdots}
\usepackage{hyperref}
\usepackage{longtable}
\usepackage[inline]{enumitem}
\newlist{choices}{enumerate*}{1}
\setlist[choices]{itemjoin = \hspace{1cm}, label=\arabic*.}

\DeclareMathOperator{\tp}{\mathfrak u}
\DeclareMathOperator{\smod}{smod}
\newcommand{\dcl}{\mathop{\mathbb N}}
\DeclareMathOperator{\id}{id}
\DeclareMathOperator{\st}{st}
\DeclareMathOperator{\PS}{PS}
\DeclarePairedDelimiter{\set}{\{}{\}}
\DeclarePairedDelimiter{\floor}{\lfloor}{\rfloor}
\DeclarePairedDelimiter{\abs}{\lvert}{\rvert}

\newcommand{\bla}[4]{{#1}_{#2}#3\ldots#3{#1}_{#4}}

\usepackage{thmtools}
\usepackage[capitalize]{cleveref}
\usepackage{nameref}
\Crefname{rem}{Remark}{Remarks}
\Crefname{alphthm}{Theorem}{Theorems}

\usepackage{centernot}
\usepackage[e]{esvect}
\newcommand{\freccia}{\vv}

\newcommand{\N}{\mathbb{N}}
\newcommand{\Q}{\mathbb{Q}}
\newcommand{\Z}{\mathbb{Z}}
\newcommand{\from}{:}
\newcommand{\inverse}{^{-1}}
\renewcommand{\star}{{}^\ast}
\renewcommand{\phi}{\varphi}
\renewcommand{\epsilon}{\varepsilon}

\theoremstyle{definition}
\newtheorem{thm}{Theorem}[section]
\newtheorem{lemma}[thm]{Lemma}
\newtheorem{oss}[thm]{Remark}
\newtheorem{rem}[thm]{Remark}
\newtheorem{teo}[thm]{Theorem}
\newtheorem{prop}[thm]{Proposition}
\newtheorem{pr}[thm]{Proposition}
\newtheorem{cor}[thm]{Corollary}
\newtheorem{co}[thm]{Corollary}
\newtheorem{fact}[thm]{Fact}
\newtheorem{defn}[thm]{Definition}
\newtheorem{defin}[thm]{Definition}
\newtheorem{esem}[thm]{Example}
\newtheorem{eg}[thm]{Example}
\newtheorem{notation}[thm]{Notation}
\newtheorem{problem}[thm]{Problem}

\newtheorem{alphthm}{Theorem}

\newtheoremstyle{named}{}{}{}{}{\bfseries}{.}{.5em}{\thmnote{#3}}
\theoremstyle{named}
\newtheorem*{namedtheorem}{Theorem}

\usepackage{titling}
\usepackage[affil-it]{authblk}
\usepackage{orcidlink}
\newcommand{\email}[1]{\href{mailto:#1}{\texttt{#1}}}

\makeatletter
\newcommand{\subjclass}[2][2020]{%
  \let\@oldtitle\@title%
  \gdef\@title{\@oldtitle\footnotetext{\hspace*{-2em}#1 \emph{Mathematics subject classification.} #2}}%
}
\newcommand{\keywords}[1]{%
  \let\@@oldtitle\@title%
  \gdef\@title{\@@oldtitle\footnotetext{\hspace*{-2em}\emph{Keywords:} #1.}}%
}
\makeatother

\author[1]{Mauro Di Nasso \orcidlink{0000-0001-6103-9775}%
  \thanks{email: \email{mauro.di.nasso@unipi.it}}}

\author[2]{Lorenzo Luperi Baglini\orcidlink {0000-0002-0559-0770}%
  \thanks{email: \email{lorenzo.luperi@unimi.it}}}

\author[1]{Marcello Mamino  \orcidlink{0000-0001-9037-5903}%
  \thanks{email: \email{marcello.mamino@unipi.it}}}

 \author[ ]{Rosario Mennuni \orcidlink{0000-0003-2282-680X}%
  \thanks{email: \email{R.Mennuni@posteo.net}}}

\author[3]{Mariaclara Ragosta \orcidlink{0009-0004-6641-4676}%
  \thanks{email: \email{mariaclara@kam.mff.cuni.cz}}}

\affil[1]{\small Dipartimento di Matematica, Universit\`a di Pisa, Largo Bruno Pontecorvo 5, 56127 Pisa, Italy}
\affil[2]{\small Dipartimento di Matematica, Università  di Milano, Via Saldini 50, 20133 Milano, Italy}
\affil[3]{\small Department of Applied Mathematics (KAM), Charles University, Malostranské náměstí 25, Praha 1, Czech
  Republic}

\title{Ramsey's Witnesses} 



\begin{document}

\maketitle

\begin{abstract}
We introduce the notion of \emph{Ramsey partition regularity}, a generalisation of partition regularity involving infinitary configurations. 
We provide characterisations of this notion in terms of certain ultrafilters 
related to tensor products and dubbed \emph{Ramsey's witnesses}; and we also consider their nonstandard counterparts
as pairs of hypernatural numbers, called \emph{Ramsey pairs}.
These characterisations are then used to determine whether various configurations involving polynomials and exponentials are Ramsey partition regular over the natural numbers. In particular, this provides negative answers to several questions recently posed by Kra, Moreira, Richter and Robertson.
\end{abstract}

\tableofcontents

{\footnotesize

MM and RM were supported by the project PRIN 2022: ``Models, sets and classifications'' Prot.\ 2022TECZJA. MDN, LLB, RM and MR were supported by the project PRIN 2022 ``Logical methods in combinatorics'', 2022BXH4R5, Italian Ministry of University and Research. We acknowledge the MIUR Excellence Department Project awarded to the Department of Mathematics, University of Pisa, CUP I57G22000700001. MR is supported by project 25-15571S of the Czech Science Foundation (GAČR). This work has been supported by Charles University Research Centre programme No.UNCE/24/SSH/026. MDN is a member of the INdAM research group GNSAGA.

}

\section{Introduction}

Since the 1970s, applications of ultrafilters to combinatorics have been widely studied,
producing a wealth of results, especially in the area of Ramsey theory; see the comprehensive monograph~\cite{hindmanAlgebraStoneCechCompactification2011}.
In recent years, the close connection between the space of ultrafilters $\beta\N$
and nonstandard extensions ${}^*\N$ of the natural numbers has been thoroughly investigated, 
providing an alternative approach that has revealed useful for formalising Ramsey properties 
in a simple and manageable way, and that has also made it possible to prove some new 
results in the area of \emph{arithmetic Ramsey theory}.

The latter is a branch of combinatorics concerned with the ``abundance'' of certain arithmetic patterns in the natural numbers. A typical problem in this area is to determine whether a given configuration is ubiquitous enough that one is forced to find instances of it in some piece of every finite partition of $\mathbb N$.

In this area, it is customary to call partitions \emph{colourings}. For instance, a classical theorem by I.~Schur~\cite{schurUberKongruenz} states that, whenever the natural numbers are coloured with finitely many colours, say $\mathbb N=\bla C1\cup r$, then one of the $C_i$ contains a triple of distinct elements $a,b,c$ such that $a+b=c$. Such a triple is called a \emph{monochromatic} solution of the equation $x+y=z$, and the existence of monochromatic solutions in all colourings is the statement that the aforementioned equation is \emph{partition regular}.

It is natural to ask which other equations have the same property. Albeit the linear case is well-understood by a classical result of R.~Rado~\cite{radoStudienZurKombinatorik1933}, the polynomial one is still very far from being settled; for instance, whether the Pythagorean equation $x^2+y^2=z^2$ is partition regular is still unknown, and one of the long-standing open problems in the area.

The notion of partition regularity is not limited to solutions to Diophantine equations. In general, a family $\mathcal{F}$ of subsets of a set $X$  is \emph{partition regular} (PR) over $X$ if it is closed under taking supersets and, for all finite colourings $X=C_{1}\cup\ldots\cup C_{r}$ there exists $i\leq r$ such that $C_{i}\in\mathcal{F}$; and it is \emph{strongly partition regular} (SPR) over $X$ if, furthermore, for all $Y\in \mathcal{F}$ the family $\mathcal{F}\cap \mathcal P(Y)$ is PR over $Y$, i.e.\ if for all finite colourings $Y=C_{1}\cup\ldots\cup C_{r}$ there exists $i\leq r$ such that $C_{i}\in\mathcal{F}$.

If $P$ is some property that can be satisfied by a set, with a small abuse we will say that the property $P$ is PR if the family $\mathcal{F}_{P}\coloneqq \{A\subseteq X\mid P \text{ is true for } A\}$ is PR. So, for example, Schur's Theorem says that the property of containing some $x,y,z$ with $x+y=z$ is partition regular. Another celebrated example is van der Waerden's Theorem~\cite{vanderwaerdenBeweisBaudetschenVermutung1927}, establishing strong partition regularity for containing arbitrarily long arithmetic progressions.

Problems in arithmetic Ramsey theory have been tackled by employing a variety of methods; as mentioned above, one that attracted significant attention in recent years is to exploit the strict connection between spaces of ultrafilters and their algebra and \emph{nonstandard extensions} of $\N$. We now briefly recall what this amounts to; see e.g.~\cite{dinassoHypernaturalNumbersUltrafilters2015, luperibagliniNonstandardCharacterisationsTensor2019, dinassoNonstandardMethodsRamsey2019} for a more detailed treatment.

The link between partition regularity of a family $\mathcal F$ of subsets of $X$ and ultrafilters on $X$ stems from the well-known facts that
\begin{enumerate}
\item $\mathcal{F}$ is PR over $X$ if and only if there exists an ultrafilter $u$ over $X$  such that $u\subseteq\mathcal{F}$; such a $u$ is called a \emph{witness} of the partition regularity of $\mathcal F$; and
\item $\mathcal{F}$ is SPR over $X$ if and only if $\mathcal F\ne \emptyset$ and for all $Y\in\mathcal{F}$ there is a witness $u\subseteq\mathcal F$ containing  $Y$.
\end{enumerate}

As for nonstandard extensions of $\mathbb N$, recall that these are \emph{elementary extensions} of the natural numbers, that is, larger structures satisfying the \emph{transfer principle} for first-order formulas.\footnote{Here $\mathbb N$ is equipped with its full first-order structure, that is, with relation symbols for every subset of each of its Cartesian powers.} For example, this means that, for every subset $A\subseteq \mathbb N^k$, there is an associated nonstandard extension $\star A\subseteq \star \mathbb N^k$ that satisfies the same first-order properties as $A$: e.g., if $A$ is the set of even natural numbers, then $\star A$ is precisely the set of those elements of $\star \mathbb N$ of the form $2\alpha$, for a suitable $\alpha\in \star \mathbb N$. Sets of the form $\star A$, and more generally \emph{internal} subsets of $\star \mathbb N$, are well-behaved in a number of senses. Crucially, they satisfy the same elementary properties as subsets of $\mathbb N$, e.g.\ every nonempty internal subset has a least element. Furthermore, they obey the \emph{overspill} principle: whenever an internal set contains arbitrarily large natural numbers, then it must also contain an element larger than $\mathbb N$; so, for instance, $\mathbb N$ is not internal.

Elements of some nonstandard extension $\star \mathbb N$ are called \emph{hypernatural numbers}, and they and their tuples bear a natural relation to ultrafilters on $\mathbb N$ and its Cartesian powers. Namely, for  $\left(\alpha_{1},\ldots,\alpha_{k}\right)\in\star{\N}^{k}$, the set\footnote{Other notations common in the literature are $u_{\left(\alpha_{1},\ldots,\alpha_{k}\right)}$ and $\operatorname{tp}(\alpha_{1},\ldots,\alpha_{k})$. }
\[\tp\left(\alpha_{1},\ldots,\alpha_{k}\right)\coloneqq \{A\subseteq \N^{k}\mid \left(\alpha_{1},\ldots,\alpha_{k}\right)\in\star{A}\}\]
belongs to $\beta\N^{k}$---that is, it is an ultrafilter on $\mathbb N^k$. Conversely, every $u\in\beta\N^{k}$ is of the form $u=\tp\left(\alpha_{1},\ldots,\alpha_{k}\right)$ for a suitable  $\left(\alpha_{1},\ldots,\alpha_{k}\right)\in\star{\N^{k}}$, provided  $\star\mathbb N$  is $\mathfrak c^+$-saturated.

Throughout the paper, we work in such a $\star \mathbb N$. We write $\left(\alpha_{1},\ldots,\alpha_{k}\right)\sim \left(\beta_{1},\ldots,\beta_{k}\right)$, and say that these tuples are \emph{equivalent}, whenever $\tp\left(\alpha_{1},\ldots,\alpha_{k}\right)=\tp\left(\beta_{1},\ldots,\beta_{k}\right)$. The equivalence classes of this relation correspond, as explained above, to ultrafilters  $u\in \beta \mathbb N^k$. For each such $u$, the associated $\sim$-equivalence class called the \emph{monad} of $u$, and its elements \emph{generators} of $u$. It is well-known that combinatorial properties of $u$ can be rephrased in terms of properties of its monad. For finitary configurations, that is, properties expressible by a first-order formula, this can be done by using the following fundamental property. See e.g.~\cite[Proposition~9.1]{dinassoNonstandardMethodsRamsey2019}, or~\cite[Theorem 5.4]{luperibagliniNonstandardCharacterisationsTensor2019} for a more general version.

\begin{fact}\label{fact:fundprop}
  A formula  $\phi(\bla x1,n)$ is PR, that is, for every finite colouring $\mathbb N=\bla C1\cup r$ there is a monochromatic $(\bla a1,n)$ such that $\mathbb N\models\phi(\bla a1,n)$, if and only if there are $\bla \alpha1\sim n$ such that $\star \mathbb N\models\phi(\bla \alpha1,n)$.
\end{fact}

Above, $\star \mathbb N\models\phi(\bla \alpha1,n)$ means that the formula $\phi(\bla \alpha1,n)$ holds in $\star \mathbb N$.
Combining this with the ultrafilter characterisation of partition regularity we see that, for instance, Schur's Theorem is, in this language, the statement that there are $\alpha\sim\beta\sim\gamma$ in $\star \mathbb N$ such that $\star \mathbb N\models\alpha+\beta=\gamma$.

Over the years Schur's result has been strengthened in various directions.
Most notably, in~\cite{hindmanFiniteSumsSequences1974} N.~Hindman proved the following infinitary extension:

\begin{namedtheorem}[Hindman's Theorem]\label{teo: Hindman} 
For every finite colouring of $\N$ there exists an infinite increasing sequence $(x_i)_{i\in \mathbb N}$
such that the set of all finite sums $\text{FS}((x_i)_{i\in \mathbb N})=\{x_{i_1}+\ldots+x_{i_k}\mid k\in \mathbb N, i_1<\ldots<i_k\}$
is monochromatic.
\end{namedtheorem}

\nameref{teo: Hindman} is the prototype of the results on the existence
of infinite monochromatic configurations, which can be seen as
leaps from the realm of finite mathematics to the world of infinite combinatorics.  

 Typically these statements are closely related, in one way or another, to Ramsey's Theorem~\cite{ramseyProblemFormalLogic1930}. 
Let us recall the $2$-dimensional case of the latter which, when applied to $X=\N$, is the case that we will study in the rest of the paper. Write $[X]^2\coloneqq \set{\set{x_1,x_2}\mid x_1, x_2\in X, x_1\ne x_2}$.

\begin{namedtheorem}[Ramsey's Theorem for pairs]\label{teo: Ramsey} Let $X$ be an infinite set. For all finite partitions $[X]^{2}=C_{1}\cup\ldots\cup C_{r}$ there exist $i\leq r$ and an infinite $H\subseteq X$ such that $[H]^2\subseteq C_i$.\end{namedtheorem}

For convenience, when $X=\N$, we identify $[H]^{2}$ with $\{(h_{1},h_{2})\in H^{2}\mid h_{1}<h_{2}\}$. 

A direct application of \nameref{teo: Ramsey} shows that the family $\mathcal F$ of those $A\subseteq \mathbb N$ such that there is an infinite $H$ with $\set{h_1+h_2\mid (h_1,h_2)\in [H]^2}\subseteq A$ is SPR. The analogous statement also holds if sums are replaced by products but, notably, sums and products cannot be combined. Indeed, a theorem proved by N.~Hindman in~\cite[Theorem~2.11]{hindmanPartitionsPairwiseSums1984a}, another proof of which recently appeared in~\cite[Theorem~3.1]{hindmanNewResultsMonochromatic2023}, states the following.

\begin{namedtheorem}[Pairwise Sum-Product Theorem]\label{thm:pwsp}
Let $\mathcal F$ be the family of those $A\subseteq \mathbb N$ such that 
 \[\exists H\subseteq \N \ H\text{ is infinite and } \forall (x,y)\in [H]^{2}\; (x+y, x\cdot y\in A).\] 
Then $\mathcal{F}$ is not PR.\end{namedtheorem}

This is (a failure of) the kind of ``Ramsey partition regularity'' statements (see \Cref{casino}) that we study in this paper, namely: for which formulas  $\varphi\left(x,y,z_{1},\ldots,z_{n}\right)$, given a finite colouring of $\mathbb N$, one can find an infinite set $H$ such that for all $(x,y)\in [H]^{2}$ there exist monochromatic $z_{1},\ldots,z_{n}$ such that $\varphi\left(x,y,z_{1},\ldots,z_{n}\right)$ holds?  E.g., in the theorem above, $\varphi(x,y,z,t)$ is the formula $(x\neq y)\wedge (x+y=z)\wedge (x\cdot y=t)$.

Because of the infinitary nature of these statements,  \Cref{fact:fundprop} does not directly apply. Nevertheless, one may still study such problems via nonstandard methods by employing the following definition, not to be confused with that of ``Ramsey ultrafilters'' (also known as ``selective ultrafilters'').

\begin{defin}\label{defin:rw} We call an ultrafilter $u\in \beta \mathbb N^2$  a \emph{Ramsey's witness} (RW) if for all $A\in u$ there exists an infinite $H\subseteq \N$ such that $[H]^{2}\subseteq A$. We also write $\mathrm{RW}$ for the subset of $\beta \mathbb N^2$ consisting of all $u$ as above, and say that $(\alpha,\beta)\in \star{\N}^{2}$ is a \emph{Ramsey pair}, if $\tp(\alpha,\beta)\in \mathrm{RW}$. In this case, we write $(\alpha,\beta)\models \mathrm{RW}$.
\end{defin}

  By \nameref{teo: Ramsey}, the family of all $A\subseteq \mathbb N^2$ such that there exists an infinite $H\subseteq \N$ with $[H]^{2}\subseteq A$ is SPR. Equivalently, every such $A$ is contained in some $u\in \mathrm{RW}$; equivalently, $\star A$ contains a Ramsey pair.

 \Cref{teo: nonstandard first order properties RW} is our main technical result. It provides a ``Ramsey'' version of \Cref{fact:fundprop},  that is, a characterisation of Ramsey partition regularity statements in terms of properties of Ramsey pairs. From this characterisation we deduce several results in the style of the \nameref{thm:pwsp}. Let us state some of them in the language we just introduced.

 \begin{alphthm}\label{alphthm:A}
  If $(\alpha,\beta)$ is a Ramsey pair, then
  \begin{enumerate}
  \item\label{point:A1} $\alpha+\beta\nsim \alpha\cdot\beta$ (\Cref{thm: no sums and products RW}),
  \item\label{point:A2} $\alpha\cdot\beta\nsim \alpha^\beta$ (\Cref{eg:yxxy}),
  \item\label{point:A3} $\alpha^\beta\nsim \beta^\alpha$ (\Cref{eg:expprod}), and
  \item\label{point:A4} $f(\alpha)\cdot2^\beta\nsim g(\alpha,\beta)$, where $f\from \mathbb N\to \mathbb N$ and $g\in \mathbb Z[x,y]$ (special case of \Cref{thm:hbfs}).
  \end{enumerate}
\end{alphthm}

Point~\ref{point:A1} is a restatement of the \nameref{thm:pwsp}, of which we provide a short nonstandard proof. It is known~\cite[Theorem~2]{sahasrabudheExponentialPatternsArithmetic2018} that every finite colouring of the naturals contains a  monochromatic quadruple $a,b,a\cdot b, a^b$. Point~\ref{point:A2} says that the Ramsey version of this does not hold, not even for $a\cdot b, a^b$, i.e., there is a finite colouring of $\mathbb N$ such that no increasing sequence $(a_i)_{i\in \mathbb N}$ is such that the set $\set{a_i\cdot a_j\mid i<j}\cup \set{a_i^{a_j}\mid i<j}$ is monochromatic. See also~\cite{dinassoMonochromaticExponentialTriples2024,nassoCentralSetsInfinite2022} for partition regularity of exponential configurations. Points~\ref{point:A3} and~\ref{point:A4} correspond to similar Ramsey properties.

 We also study the Ramsey partition regularity of polynomial equations. Namely, given a nonzero polynomial $P \in \mathbb Z[\bla x1,n]$, we investigate whether the formula $P(\bla x1,n)=0$ is Ramsey PR. In the two-variable case, there are no Ramsey PR equations (\Cref{rem:no2varRPR}). Our next result concerns certain $3$-variable polynomials.

 \begin{alphthm}[\Cref{cor: poly applications,pr:ramseydiff}]\label{alphthm:B}
  Let $P,Q\in \Z[x]$ be nonzero polynomials with $P(0)=Q(0)=0$ and $a,n\in\N$. The following are equivalent.
  \begin{enumerate}
  \item There are $\alpha\sim \beta\sim \gamma$ with $(\alpha,\beta)$ a Ramsey pair such that $a\alpha^{n}+P(\beta)=Q(\gamma)$;
  \item $ax^n+P(y)-Q(z)$ is either $a(x+y-z)$ or $a(x-y+z)$.
  \end{enumerate}
\end{alphthm}

In other words, up to multiplying everything by a constant, the only two equations of the form $ax^n+P(y)=Q(z)$ that are Ramsey PR in $x,y$ are Schur's equation $x+y=z$ and the equation $x-y=-z$. An interesting property of the latter is that, contrarily to Schur's equation, it cannot be solved by tensor pairs (\Cref{rem:noETdifference}). \Cref{alphthm:B} rules out the Ramsey partition regularity of plenty of equations, ranging from some whose partition regularity is open, such as the Pythagorean equation $x^2+y^2=z^2$, to  equations of the form $y-x=z^n$, for $n\ge 2$, that are known to be PR, see \cite[Theorem C]{BFMC96}.

Similar Ramsey partition regularity problems were posed in~\cite{kraProblemsInfiniteSumset2025}, which contains questions arising from the solution~\cite{moreiraProofSumsetConjecture2019} to Erd\H os' $B+C$ conjecture, arguably the most important recent result in combinatorics. In particular, \Cref{alphthm:B} answers in the negative several questions from Sections~3.3 and~3.4 of~\cite{kraProblemsInfiniteSumset2025}, in the stronger PR version, albeit in the form without shifts.

For instance, in the language of the present work, the shiftless PR version of \cite[Question~3.15]{kraProblemsInfiniteSumset2025} amounts to asking whether the equation $x^2+y=z$ is Ramsey PR. In~\cite[Theorem~1.5]{ackelsbergInfinitePolynomialPatterns2025} it is shown that this equation is Ramsey PR over $\mathbb Q$, and in fact that the same holds with $x^2$ replaced by any $P(x)\in \mathbb Q[x]$ of degree at least $2$. 
This fails over $\mathbb N$ by \Cref{alphthm:B} and, after a preprint version of this paper was circulated, E.~Ackelsberg gave an explicit $5$-colouring witnessing that $x^2+y=z$ is not Ramsey PR,  see~\cite[Section~11.1]{ackelsbergInfinitePolynomialPatterns2025}.
Analogously, \Cref{alphthm:B} answers in the negative the PR shiftless versions of~\cite[Questions 3.10, 3.11, 3.12, 3.15, 3.19, 3.20]{kraProblemsInfiniteSumset2025}, as well as the variant of~\cite[Question~3.16]{kraProblemsInfiniteSumset2025} with monomials on the left hand side but arbitrary polynomials without constant term on the right hand one.

\Cref{alphthm:A,alphthm:B} demonstrate that being Ramsey PR significantly strengthens being merely PR. One may wonder what happens to other properties classically known to be PR. In this direction, another application of our results is a proof that van der Waerden's Theorem does not admit a ``Ramsey version'', not even in the case of $3$-term arithmetic progressions.

\begin{alphthm}[\Cref{no infinite 3-AP}]\label{alphthm:C}
  If $\alpha\sim\beta\sim\gamma$ and $(\alpha,\beta)$ is a Ramsey pair, then $\alpha,\beta,\gamma$ do not form, in any order, an arithmetic progression.
\end{alphthm}

The paper is structured as follows. In  \Cref{Section: RW} we will prove several properties of Ramsey pairs and characterise them in terms of tensor pairs. In  \Cref{Section: main defn}, we introduce several versions of Ramsey partition regularity and characterise them in ultrafilter and nonstandard terms. In  \Cref{sums and products} we study these notions for the formula $\varphi(x,y,z,t)$ that appears in the \nameref{thm:pwsp}, and give a very compact proof of the latter. In  \Cref{pr poly} we prove \Cref{alphthm:A,,alphthm:B,,alphthm:C}, and in our final  \Cref{Section: open problems} we collect some open problems that we believe deserve further study.

\paragraph{Acknowledgements} We thank the anonymous referees for their careful reading of this manuscript and for many useful suggestions. We thank Ethan Ackelsberg for useful conversations around \Cref{alphthm:B}.

\section{Ramsey's witnesses}\label{Section: RW}

In this paper, $\mathbb N$ denotes the set of \emph{positive} integers. We assume familiarity with the basics of nonstandard methods, for which we refer the reader to, e.g., \cite{goldblattLecturesHyperreals1998}, but we briefly recall some well-known facts on tensor pairs.

The relationship between ultrafilters on $\N$ and on $\N^{2}$ is intricate, as the product filter $u\times v$ of two ultrafilters on $\N$, generated by the family $\set{A\times B\mid A\in u, B\in v}$, is not in general an ultrafilter on $\N^2$. Among all the ultrafilters that extend $u\times v$, one plays a special role in many applications, namely the tensor product $u\otimes v$.

\begin{defin} Given $u,v\in\beta\N$, the \emph{tensor product} $u\otimes v$ is the ultrafilter on $\N^{2}$ defined as follows: for all $A\subseteq \N^{2}$, we let
\[A\in u\otimes v\Leftrightarrow \{n\in\N\mid \{m\in\N\mid (n,m)\in A\}\in v\}\in u.\]
Given $(\alpha,\beta)\in\star\N^{2}$, we say that $(\alpha,\beta)$ is a \emph{tensor pair} if $\tp(\alpha,\beta)$ is a tensor product. 
\end{defin}

The tensor product can also be considered in arbitrary dimensions and is associative, hence it makes sense to talk of tensor triples, tensor $k$-uples, etc. However, in this paper we will mostly work with tensor products in dimension 2. Note that, almost by definition, if $(\alpha,\beta)$ is a tensor pair then $\tp(\alpha,\beta)=\tp(\alpha)\otimes\tp(\beta)$. For a detailed study of tensor products we refer to~\cite{luperibagliniNonstandardCharacterisationsTensor2019}. In more model-theoretic language, $(\alpha,\beta)$ is a tensor pair if the type of $\alpha$ over $\mathbb N\cup \set\beta$ is a \emph{coheir} of its restriction to $\mathbb N$, that is, is finitely satisfiable in $\mathbb N$, which amounts to the following.
\begin{fact}\label{fact:coheir}
The pair $(\alpha,\beta)$ is tensor if and only if for every $A\in \tp(\alpha,\beta)$ there is $n\in \mathbb N$ such that $A\in \tp(n,\beta)$.
\end{fact}
This point of view is exploited for example in \cite{collaRAMSEYSCOHEIRS2022}.

Combinatorial applications of tensor pairs by means of nonstandard methods are based on their characterisation given by C.~Puritz in \cite[Theorem~3.4]{puritzSkiesConstellationsMonads1972} (for higher dimensions, see \cite[Theorem~4.22]{luperibagliniNonstandardCharacterisationsTensor2019}\footnote{In the cited result, there is a missing assumption that $\bla \alpha1,n$ are not standard.}). Before recalling Puritz' Theorem, let us review some prerequisites. See e.g.~\cite{dinassoTasteNonstandardMethods2014,dinassoHypernaturalNumbersUltrafilters2015,dinassoNonstandardMethodsRamsey2019,luperibagliniNonstandardCharacterisationsTensor2019} for further details.

\begin{notation}
  Throughout the paper we treat functions $f\from \mathbb N^k\to \mathbb N$ as symbols in the language of the full structure on $\mathbb N$; in practice this means that, if $\alpha\in \star \mathbb N^k$, its image under the nonstandard extension of $f$ to $\star \mathbb N^k$ will be denoted by $f(\alpha)$ instead of $\star f(\alpha)$. Similarly, we use the symbol $f$ also for the induced map $\beta \mathbb N^k\to \beta \mathbb N$ (which is well-defined by \Cref{fact:simprop} below). We write $\dcl(\alpha)$ for the definable closure of $\alpha$, namely
  \[
    \dcl(\alpha)\coloneqq\set{f(\alpha)\mid f\from \mathbb N \to \mathbb N}.
  \]
\end{notation}

\begin{fact}\label{fact:simprop}Let $\alpha\in \star \mathbb N^n$, $\beta\in \star \mathbb N^m$, $f\from \mathbb N^n\to \mathbb N^l$, and $g\from\N^{m}\rightarrow\N^{l}$.
  \begin{enumerate}
  \item\label{point:simpushf}   If $\alpha\sim \beta$ then  $f(\alpha)\sim f(\beta)$.
\item\label{point:fideq}        If   $f(\alpha)\sim \alpha$ then $f(\alpha)=\alpha$.
    \item The following are equivalent.
\begin{enumerate}
\item $f(\alpha)\sim g(\beta)$.
\item There exists $\beta'\sim \beta$ such that $f(\alpha)= g(\beta')$.
\end{enumerate}
  \item If $\alpha,\beta\in \star \mathbb N$,  $\alpha\sim \beta$, and $\alpha<\beta$, then $\beta-\alpha$ is infinite.
  \end{enumerate}
\end{fact}
\begin{fact}[Puritz]\label{puritz} Let $(\alpha,\beta)\in\star{\N^{2}}$. Then $(\alpha,\beta)$ is a tensor pair if and only if for all $f\from\N\rightarrow \N$ we have $f(\beta)\in\N$ or $\alpha<f(\beta)$.\end{fact}

The following remarks are well-known, and can be easily proven by using \Cref{puritz}.

\begin{rem}\label{rem:largerthandcl}
  If $(\alpha,\beta)$ is a tensor pair and $\beta\notin \mathbb N$, then $\beta>\alpha$.
\end{rem}

\begin{rem}\label{rem:tensorpshfwd}
  If $(\alpha,\beta)$ is a tensor pair, and $f,g\from \mathbb N\to \mathbb N$, then $(f(\alpha),g(\beta))$ is a tensor pair.
\end{rem}

In this paper, we focus on the relationship between tensor products and Ramsey's witnesses, which is grounded on the following fact. Its proof is implicit in the usual ultrafilter proof of \nameref{teo: Ramsey}, see e.g.~\cite[Theorem~3.3.7]{changModelTheory1990}; below, we make it explicit for the reader's convenience.

\begin{pr}\label{thm:tensors witness Ramsey Theorem}  Let $u\in \beta\N$ be a nonprincipal ultrafilter. Then $u\otimes u\in \mathrm{RW}$.
\end{pr}

\begin{proof}
  For $A\subseteq \mathbb N^2$ and $n\in \mathbb N$, write $A_n\coloneqq\set{m \mid (n,m)\in A}$, and let $A_u\coloneqq\set{n\mid A_n\in u}$. By definition, $A\in u\otimes u\iff A_u\in u$. Fix  $A\in u\otimes u$. We construct inductively an increasing sequence $(x_i)_{i\in \mathbb N}$ such that, for each $i$, we have $A_{x_i}\in u$, and if $i<j$ then $(x_i, x_j)\in A$. To construct $x_{k+1}$, given $\bla x1<k$, we observe that the set $A_u\cap\bigcap_{i\le k} A_{x_i}$ belongs to $u$, hence is infinite, so it contains some $x_{k+1}>x_k$.
\end{proof}

Therefore, all products $u\otimes u$ with $u$ nonprincipal are Ramsey's witnesses. Actually, we can be more precise: the set $\mathrm{RW}$ is the closure of the set of such products. Recall that the topology on $\beta \mathbb N^k$ has as basic clopen sets those of the form $\set{u\in \beta \mathbb N^k\mid A\in u}$, for $A\subseteq \mathbb N^k$.

\begin{defin} 
We define the set of \emph{tensor squares} as $\mathrm{TS}\coloneqq \{u\otimes u\in\beta\N^{2}\mid u\in \beta\N\setminus \mathbb N\}$. We adopt similar conventions as in \Cref{defin:rw}, e.g.\ write $(\alpha,\beta)\models\mathrm{TS}$ when $\tp(\alpha,\beta)\in \mathrm{TS}$.
\end{defin}

\begin{thm}\label{thm: RW and tensor}  The topological closure  $\overline{\mathrm{TS}}$ coincides with $\mathrm{RW}$. \end{thm}

\begin{proof} It is clear from \Cref{defin:rw} that $\mathrm{RW}$ is a closed subset of $\beta\mathbb N^2 $, and it contains every $u\otimes u$ by \Cref{thm:tensors witness Ramsey Theorem}, so it clearly includes $\overline{\mathrm{TS}}$. Conversely, let $u\in\mathrm{RW}$, fix $X\in u$, and
choose an infinite $H\subseteq \mathbb N$ with $[H]^2\subseteq X$. For $h\in H$ let $H_h=\{x\in H\mid x>h\}$; the family $\{H_h\mid h\in H\}$ has the FIP, and furthermore every finite intersection of this family is infinite. Therefore, it can be extended to a nonprincipal ultrafilter $v$. For every $h\in H$ the fiber $([H]^2)_h=H_h$ belongs to $v$, and so $H=\{n\in\N\mid ([H]^2)_n\in v\}\in v$, i.e.\ $X\supseteq[H]^2\in v\otimes v$. \end{proof}

We will need the following slight strengthening of \Cref{fact:simprop}(\ref{point:fideq}). 

\begin{pr}\label{pr:bddto1}
Let $f,g\from \mathbb N^n\to \mathbb N^l$. If $g$ is bounded-to-one, i.e.\ there is $k\in \mathbb N$ such that every fiber of $g$ has size at most $k$, and $f(\alpha)\sim g(\alpha)$, then $f(\alpha)=g(\alpha)$.
\end{pr}
\begin{proof}
  Without loss of generality $\alpha\notin \mathbb N^n$, otherwise the conclusion trivially holds.

If $g$ is bijective, then by \Cref{fact:simprop}(\ref{point:simpushf}) $\alpha\sim g\inverse(f(\alpha))$, so it suffices to apply \Cref{fact:simprop}(\ref{point:fideq}) to $g\inverse\circ f$. The general case can be reduced to that of a bijection as follows.
  
For $i\le k$, let $C_i\subseteq \mathbb N^n$ be the set of those $a\in \mathbb N^n$  that are the $i$-th point of $g\inverse(g(a))$ in the lexicographical order. Let $i$ be such that $\alpha\in \star C_i$. By repeatedly splitting $C_i$ and selecting the piece belonging to $\tp(\alpha)$, we now find some  $C\in \tp(\alpha)$ such that $C\subseteq C_i$ and all of $C_i\setminus C$, $g(C)$ and $g(\mathbb N^n\setminus C)$ are infinite. By a routine argument, there is a bijective $h$ that agrees with $g$ on $C$, and in particular $h(\alpha)=g(\alpha)$, which reduces the problem to the case above.
\end{proof}

\begin{pr}\label{pr:ET}
   Let $\alpha,\beta\in\star{\N}\setminus\N$. The following properties hold.
   \begin{enumerate}
   \item\label{point:RW1} $(\alpha,\beta)\models \mathrm{RW}$ if and only if, for all $A\subseteq \N^{2}$, if $(\alpha,\beta)\in\star{A}$ then there exists a pair $(\gamma,\delta)\in\star{A}$ with $(\gamma,\delta)\models \mathrm{TS}$.
   \item\label{point:ET3} If $(\alpha,\beta)\models\mathrm{TS}$  and $f\from \N\rightarrow \N$ is such that $f(\alpha)\ne f(\beta)$ then $(f(\alpha),f(\beta))\models\mathrm{TS}$.
   \item \label{point:ET7}If $(\alpha,\beta)\models\mathrm{TS}$ and $g\from \N\rightarrow\N$ is such that $g(\beta)\ne g(\alpha)$, then for all $f\from \N\rightarrow\N$ $g(\beta)>f(\alpha)$.
   \end{enumerate}
\end{pr}

\begin{proof}
  \begin{enumerate}[wide=0pt]
  \item  This is just a reformulation of  \Cref{thm: RW and tensor} in terms of generators.
  \item By \Cref{rem:tensorpshfwd} $(f(\alpha),f(\beta))$ is a tensor pair, and by \Cref{fact:simprop}(\ref{point:simpushf}) $f(\alpha)\sim f(\beta)$. This, together with $f(\alpha)\ne f(\beta)$, implies $f(\alpha)\notin \mathbb N$, hence the conclusion.
  \item By \Cref{rem:tensorpshfwd} $(f(\alpha), g(\beta))$ is a tensor pair.  As $g(\beta)\ne g(\alpha)$ and $\beta\sim\alpha$, we must have $g(\beta)\notin \mathbb N$, hence we conclude by \Cref{rem:largerthandcl}.\qedhere
  \end{enumerate}
\end{proof}

\begin{pr}[Properties of Ramsey pairs]\label{co:RWbasics}
   Let $(\alpha,\beta)\models \mathrm{RW}$ and $f,g\from \mathbb N\to \mathbb N$. The following properties hold.
   \begin{enumerate}
	\item\label{point:RWfrw} If $f(\alpha)\neq f(\beta)$ then $(f(\alpha),f(\beta))\models\mathrm{RW}$.
        \item \label{point:RW7}If  $g(\alpha)\ne g(\beta)$, then $f(\alpha)<g(\beta)$.
        \item \label{point:RWWP} Either $f(\alpha)=f(\beta)$ or $\alpha<f(\beta)$.
        \item \label{point:RW8} $f(\alpha)<\beta$. In particular $\alpha<\beta$.
        \item If $P,Q\in\Z[x]$, with $Q$ of degree at least $1$, then $\abs{P(\alpha)}<\abs{Q(\beta)}$.
        \item\label{point:RWfintoone} If $g$ is finite-to-one, then $f(\alpha)<g(\beta)$.
        \item\label{point:RW5}  $\alpha\sim \beta$.
   \end{enumerate}
\end{pr}

\begin{proof}
  \begin{enumerate}[wide=0pt]
    \item Fix $A\in\tp(f(\alpha),f(\beta))$ and let $B=(f,f)^{-1}(A)\cap\{(a,b)\mid f(a)\neq f(b)\}$. By assumption, $(\alpha,\beta)\in \star B$, so by \Cref{pr:ET}(\ref{point:RW1}) there exists $(\gamma,\delta)\in\star B$ with $(\gamma,\delta)\models \mathrm{TS}$. As $f(\gamma)\neq f(\delta)$, the pair $(f(\gamma),f(\delta))\in\star{A}$ satisfies $(f(\gamma),f(\delta))\models \mathrm{TS}$ by   \Cref{pr:ET}(\ref{point:ET3}). We conclude by \Cref{pr:ET}(\ref{point:RW1}).
    \item  For fixed $f,g$, the set $\set{(x,y)\mid g(x)\ne g(y)\to f(x)<g(y)}$ belongs to every element of $\mathrm{TS}$ by \Cref{pr:ET}(\ref{point:ET7}). The conclusion then follows by  \Cref{thm: RW and tensor}.
    \item Immediate from~\ref{point:RW7}.
  \item  This is a special case of~\ref{point:RW7}.
  \item This too is a special case of~\ref{point:RW7}.

  \item This follows by applying~\ref{point:RW8} to the function $h(n)\coloneqq \max g\inverse([1,f(n)])$, which is well-defined because $g$ is finite-to-one.
      \item  Equivalence of $\alpha$ and $\beta$ is a closed condition and holds on $\mathrm{TS}$, so we conclude by \Cref{thm: RW and tensor}.\qedhere
  \end{enumerate}
\end{proof}

 \Cref{thm: RW and tensor} allows us to give a nonstandard description of $\mathrm{RW}$, \Cref{thm:rwprojinttens} below, that makes use of internal ultrafilters, that is, elements of $\star(\beta \mathbb N^k)$.

\begin{defin} Let $\mathfrak u\in \star(\beta\N^{2})$. We let $\st(\mathfrak u)\coloneqq \{A\subseteq\N^{2}\mid \star{A}\in \mathfrak u\}\in\beta\N^{2}$.\end{defin}

\begin{rem} The notation $\st(\mathfrak u)$ is consistent with viewing $\star(\beta\N^{2})$ topologically as a nonstandard extension of $\beta\N^2$, namely $\st(\mathfrak u)=u$ if and only if $\mathfrak u\in \star{A}$ for all neighbourhoods $A$  of $u$. \end{rem}

Below, we abuse the notation and write $\otimes$ for the nonstandard extension of $\otimes$ to $\star(\beta \mathbb N^2)$.

\begin{thm}\label{thm:rwprojinttens}
  $v\in \mathrm{RW}$ if and only if there is $\mathfrak u\in \star(\beta \mathbb N\setminus \mathbb N)$ such that $v=\st(\mathfrak u\otimes \mathfrak u)$.
\end{thm}

\begin{proof} This is a special case of the nonstandard characterisation of topological closure, but we spell out some details for the reader's convenience.

Left to right, for $X\subseteq \mathbb N^2$, define
  \[
    \Lambda_X\coloneqq\set{u\in \beta \mathbb N\setminus \mathbb N\mid X\in u\otimes u}.
  \]
 By assumption and \Cref{thm: RW and tensor}, for each $X\in v$ the set $\Lambda_X$ is nonempty. In fact, the family $\set{\Lambda_X \mid X\in v}$ has the FIP, and it follows that there is $\mathfrak u\in \star(\beta \mathbb N\setminus \mathbb N)$ such that, for every $X\in v$ we have  $\star X\in \mathfrak u\otimes \mathfrak u$. This implies that $\st(\mathfrak u\otimes \mathfrak u)=v$.

  Right to left, let $\mathfrak u\in \star(\beta \mathbb N\setminus \mathbb N)$. By transfer, for every internal $B\subseteq \star \mathbb N^2$ such that $B\in \mathfrak u\otimes \mathfrak u$, there is a $\star$infinite $H\subseteq \star \mathbb N$ such that $[H]^2\subseteq B$. In particular, this applies to those $B$ of the form $\star X$ for $X\subseteq \mathbb N^2$. By applying downward transfer to the formula ``there is a $\star$infinite $H\subseteq \star \mathbb N$ such that $[H]^2\subseteq \star X$'', we find an infinite $K\subseteq \mathbb N$ such that $[K]^2\subseteq X$. It follows that $\st(\mathfrak u\otimes \mathfrak u)\in\mathrm{RW}$.
\end{proof}

\begin{pr}\label{co:intfunrw}
If $(\alpha,\beta,\gamma)$ is a tensor triple, $\beta\sim \gamma$, and $f\from \mathbb N^2\to \mathbb N$, is such that $f(\alpha,\beta)\ne f(\alpha,\gamma)$, then $(f(\alpha,\beta), f(\alpha,\gamma))\models \mathrm{RW}$.
\end{pr}

\begin{proof}
Given $A\in \tp(f(\alpha,\beta),f(\alpha,\gamma))$, we need to find an element of $\mathrm{TS}$ containing $A$. By associativity, the pair $(\alpha, (\beta,\gamma))$ is tensor, hence by \Cref{fact:coheir} and assumption there is $n\in \mathbb N$ such that $A\in \tp(f(n,\beta),f(n,\gamma))$ and $f(n,\beta)\ne f(n,\gamma)$. The latter, together with $(\beta,\gamma)\models\mathrm{TS}$, gives us by \Cref{pr:ET}(\ref{point:ET3}) that $(f(n,\beta),f(n,\gamma))\models\mathrm{TS}$, hence the conclusion.
\end{proof}

 \Cref{puritz} combined with \Cref{co:intfunrw} allows us to build explicit Ramsey pairs that are not tensor pairs.

\begin{eg}\label{Rosario} Let $u,v\in\beta\N$ be non-principal and let $w\in\beta\N^{2}$ be the image (or pushforward) of $u\otimes v\otimes v$ under the function $h\from \mathbb N^3\to \mathbb N^2$ defined as $h(x,y,z)=(2^x3^y, 2^x3^z)$. That is, for all $A\subseteq \N^{2}$ \[A\in w\Leftrightarrow \{(a,b,c)\mid (2^{a}3^{b},2^{a}3^{c})\in A\}\in u\otimes v\otimes v.\]
  To check that $w\in \mathrm{RW}$, 
let $(\alpha,\beta,\gamma)$ be a generator of $u\otimes v\otimes v$; so $(\alpha,\beta,\gamma)$ is a tensor triple and $\beta\sim \gamma$. Therefore, by \Cref{co:intfunrw} we have $(2^{\alpha}3^{\beta},2^{\alpha}3^{\gamma})\models\mathrm{RW}$, and by \Cref{puritz} applied with the $2$-adic valuation we obtain $(2^{\alpha}3^{\beta},2^{\alpha}3^{\gamma})\centernot\models\mathrm{TS}$.
\end{eg}

\begin{rem}
By \Cref{Rosario}  the inclusion $\mathrm{TS}\subseteq \mathrm{RW}$ is strict, hence $\mathrm{TS}$ is not closed.
\end{rem}

\section{Ramsey partition regularity}\label{Section: main defn}

Throughout this section, we will consider formulas with multiple variables playing different roles; all of these are first order formulas in the full language on $\N$. To avoid handling a terrible amount of indices, and to increase the readability, we will use the notation $\freccia{x}$ to denote a list of variables of some length $l(\freccia{x})$, whose $i$-th variable, for $i\leq l(\freccia{x})$, will be denoted by $(\freccia{x})_{i}$; namely, $\freccia{x}=\left((\freccia{x})_{1},\ldots,(\freccia{x})_{l(\freccia{x})}\right)$. We explicitly allow the degenerate case $l(\freccia x)=0$. Similar notations will be used for lists of numbers; in this case, we will write $\freccia{a}\in A$ as a shortened notation for $\freccia a\in A^{l(\freccia a)}$.
 
The most studied problem in arithmetic Ramsey theory regards the partition regularity of formulas. Often, a related notion that is considered is that of partial partition regularity:

\begin{defin} A formula $\varphi\left(\freccia{x},\freccia{y}\right)$ is \emph{partially partition regular} (PPR) in the variables $\freccia{x}$ if for all finite partitions $\N=C_{1}\cup\ldots\cup C_{r}$ there exist $i\leq r$, $\freccia{a}\in C_{i}$, and $\freccia{b}\in\N$ such that $\varphi(\freccia{a},\freccia{b})$ holds. \end{defin}

\begin{eg}\label{eg:pythPPR}
As recently shown in~\cite{frantzikinakisPartitionRegularityPythagorean2025}, the formula $x^{2}+y^{2}=z^{2}$ is PPR in the variables $x,y$.
\end{eg}
Partial partition regularity can be seen as a particular case of the following general definition:

\begin{defin}\label{def: sep PR} A formula $\varphi(\freccia{x_1},\ldots,\freccia{x_n})$ is \emph{block partition regular} in the tuples of variables $\freccia{x_1},\ldots,\freccia{x_n}$ (notation: BPR in $\freccia{x_1}\mid \ldots\mid \freccia{x_n}$) if for all finite partitions $\N=C_{1}\cup\ldots\cup C_{r}$, for every $j\leq n$ there exist $i_{j}\leq r$ and $\freccia{a_j}\in C_{i_{j}}$ such that $\varphi\left(\freccia{a_1},\ldots,\freccia{a_n}\right)$.\end{defin}

PR is the case where $n=1$, PPR in $\freccia{x_1}$ is the case where $l(\freccia{x_i})=1$ for all $i\geq 2$.

\begin{oss}\label{party} In  \Cref{def: sep PR} a partition of the variables of $\varphi$ is given; it is trivial to observe that this notion is preserved if the partition is refined, namely: if the variables appearing in $\freccia{x_i}$ are partitioned in $\freccia{y_{i,1}},\freccia{y_{i,2}}$ and $\varphi$ is BPR in $\freccia{x_1}\mid \ldots\mid \freccia{x_n}$, then it is BPR in $\freccia{x_1}\mid \ldots\mid \freccia{x_{i-1}}\mid \freccia{y_{i,1}}\mid \freccia{y_{i,2}}\mid \freccia{x_{i+1}}\mid \ldots\mid \freccia{x_{n}}$.\end{oss}

\begin{rem}
It is easy to see that  \Cref{fact:fundprop} generalises to block partition regularity as follows. A formula $\phi$ is BPR in $\bla {\freccia x}1\mid n$ if and only if, for $i\le n$, there are tuples $\freccia{\alpha_i}\in \star \mathbb N$ such that, for every $i$, we have $\bla {(\freccia{\alpha_i})}1\sim{l(\freccia{\alpha_i})}$.
\end{rem}

In this paper, we are not going to study general properties of block partition regularity; in fact, as said in the introduction, we are mainly interested to use Ramsey pairs to study ``infinitary statements'' in combinatorics close in spirit to \nameref{teo: Ramsey}. ``Infinitary statements'' is a vague concept, that could be formalised in several different ways; in this paper, we define it as the following strengthening of block partition regularity:

\begin{defin}\label{casino} We say that $\varphi\left(x,y,\freccia{z_{1}},\ldots, \freccia{z_{n}}\right)$ is \emph{Ramsey partition regular} (Ramsey PR) in $(x,y),\freccia{z_{1}}\mid\ldots\mid\freccia{z_{n}}$ if for all finite colourings $\N=C_{1}\cup\ldots\cup C_{r}$ there are a colour $i_1\le r$, an infinite set $H\subseteq C_{i_{1}}$, and colours $\bla i2,n$ such that the following holds: for all $h_{1}<h_{2}\in H$ there exist  $\freccia{a_1}\in C_{i_1}$, \ldots, $\freccia{a_n}\in C_{i_n}$ satisfying $\varphi\left(h_{1},h_{2},\freccia{a_{1}},\ldots,\freccia{a_{n}}\right)$.
\end{defin}

\begin{oss} Let us make a few comments about the above definition.
\begin{enumerate}
\item If $\varphi\left(x,y,\freccia{z_{1}},\ldots, \freccia{z_{n}}\right)$ is Ramsey PR in $(x,y),\freccia{z_{1}}\mid\ldots\mid\freccia{z_{n}}$, then it is in particular BPR in $x,y,\freccia{z_{1}}\mid\ldots\mid\freccia{z_{n}}$
\item In \Cref{casino}, $\freccia{z_{1}}$ plays a special role: when searching for ``block monochromatic witnesses'' of $\varphi\left(x,y,\freccia{z_{1}},\ldots, \freccia{z_{n}}\right)$, we search for them with $x,y,\freccia{z_{1}}$ monochromatic; in particular, when $n=1$, if the formula $\varphi\left(x,y,\freccia{z_{1}}\right)$ is Ramsey PR in $(x,y),\freccia{z_1}$ then it is in particular PR, but much more is true. For example, the fact that the formula $x+y=z$ is PR is the content of Schur's Theorem, whilst the fact that $x+y=z$ is Ramsey PR in $(x,y),z$ amounts to the fact that, if we denote $\PS(H)\coloneqq\{h_{1}+h_{2}\mid h_{1}<h_{2}\in H\}$, then for every finite colouring of $\mathbb N$ there are an infinite set $H$ and a colour $C$ such that $H\cup \PS(H)\subseteq C$.

\item In our definition, it is allowed to have $l(\freccia{z_{1}})=0$. In this case,  \Cref{casino} could be slightly simplified: the request of monochromaticity on $H$ could be dropped, as if an infinite set $H$ with the property of  \Cref{casino} exists, then a monochromatic infinite set with the same property exists by the pigeonhole principle.
\item The analogue of  \Cref{party} holds also for Ramsey partition regularity, provided one does not split $(x,y)$ into separate blocks.
\end{enumerate}
\end{oss}

In \Cref{sec:casestudy}, we will see which of the above partition regularity notions are satisfied by an important and well-studied example, namely the formula 
\[
  \varphi(x,y,z,t)\coloneqq (x+y=z)\wedge (x\cdot y=t) \wedge x\ne y.
\]

As recalled in the introduction, partition regularity problems can be rephrased both in ultrafilters and in nonstandard terms. The main result of this section is the analogous characterisation of Ramsey partition regularity, which will be central for our applications.

\begin{teo}\label{teo: nonstandard first order properties RW}
    Let $\varphi\left(x,y,\freccia{z_{1}},\ldots, \freccia{z_{n}}\right)$ be given. The following are equivalent:

    \begin{enumerate}
		\item\label{point:RWchardef} $\varphi\left(x,y,\freccia{z_{1}},\ldots, \freccia{z_{n}}\right)$ is Ramsey PR in $(x,y),\freccia{z_{1}}\mid \ldots\mid \freccia{z_{n}}$; that is, for all finite colourings $\N=C_{1}\cup\ldots\cup C_{r}$ there are a colour $i_1\le r$, an infinite set $H\subseteq C_{i_{1}}$, and colours $\bla i2,n$ such that for all $h_{1}<h_{2}\in H$ there exist  $\freccia{a_1}\in C_{i_1}$, \ldots, $\freccia{a_n}\in C_{i_n}$ satisfying $\varphi\left(h_{1},h_{2},\freccia{a_{1}},\ldots,\freccia{a_{n}}\right)$.
        
                \item \label{point:RWchardefprime}For all finite colourings $\N=C_{1}\cup\ldots\cup C_{r}$ there is an infinite set $H$ such that for all $h_{1}<h_{2}\in H$ there exist colours $\bla i1,n$ with  $h_1, h_2\in C_{i_1}$, and there exist $\freccia{a_1}\in C_{i_1}$,\ldots, $\freccia{a_n}\in C_{i_n}$ satisfying $\varphi\left(h_{1},h_{2},\freccia{a_{1}},\ldots,\freccia{a_{n}}\right)$.
                          		\item\label{point:RWchargen} There is $(\alpha, \beta)\models\mathrm{RW}$ and there are $\freccia{\gamma_{1}},\ldots,\freccia{\gamma_{n}}\in\star{\N}$ such that 
				\begin{itemize}
				\item for all $i\leq n$, for all $h,k\leq l(\freccia{z_{i}})$ we have $(\freccia{\gamma_{i}})_{h}\sim(\freccia{\gamma_{i}})_{k}$;
				\item if $l(\freccia{z_{1}})\geq 1$ then $\alpha\sim(\freccia{\gamma_{1}})_{1}$; and
				\item $\varphi\left(\alpha,\beta,\freccia{\gamma_{1}},\ldots, \freccia{\gamma_{n}}\right)$ holds.
				\end{itemize}
	 \item\label{point:RWcharultra} There exist $u\in\mathrm{RW}\subseteq\beta\N^{2}$ and $u_{1},\ldots,u_{n}\in \beta\N$ such that 
	\begin{itemize}
	\item $\pi_{1}(u)=u_{1}$, where $\pi_1$ denotes the projection on the first coordinate; and
	\item for all $A\in u, B_{1}\in u_{1},\ldots,B_{n}\in u_{n}$ there exists an infinite set $H\subseteq B_1$ such that $[H]^{2}\subseteq A$ and for all $h_{1}<h_{2}\in H$ there are $\freccia{b_{1}}\in B_{1}, \ldots, \freccia{b_{n}}\in B_{n}$ such that $\varphi(h_{1},h_{2},\freccia{b_{1}},\ldots,\freccia{b_{n}})$ holds.
	\end{itemize}
    \end{enumerate}
\end{teo}

\begin{proof}Left to right  in  $\ref{point:RWchardef}\Leftrightarrow\ref{point:RWchardefprime}$ is clear, as in \ref{point:RWchardef} the colours $\bla i1,n$ do not depend on $h_1,h_2$, while in $\ref{point:RWchardefprime}$ they may, and right to left is a standard application of \nameref{teo: Ramsey}.

  $\ref{point:RWchardef}\Rightarrow\ref{point:RWchargen}$ Given a finite partition $P$ of $\mathbb N$, say that a tuple of natural numbers is \emph{$P$-monochromatic} if they all belong to the same piece of $P$. Define
\begin{align*}\Lambda_{P}\coloneqq\Bigl\{(h_1,h_2,\freccia{z_{1}},\ldots,\freccia{z_{n}})\Bigm| h_1<h_2&\wedge                                                                    \varphi\left(h_{1},h_{2},\freccia{z_{1}},\ldots,\freccia{z_{n}}\right)\wedge \\
                                                                &\wedge\left(h_{1},h_{2},\freccia{z_{1}} \text{ is $P$-monochromatic}\right)  \\
                                                                &\wedge \left(\freccia{z_{2}} \text{ is $P$-monochromatic}\right)\\
&\wedge\ldots\\
&  \wedge\left(\freccia{z_{n}} \text{ is $P$-monochromatic}\right)\Bigr\} 
.\end{align*}

The family $\set{\Lambda_{P}\mid{P\text{ partition of } \mathbb N}}$, ordered by inclusion, is downward directed because $\bigcap_{i\le \ell}\Lambda_{P_i}$ contains $\Lambda_Q$, where $Q$ is the common refinement of the partitions $\bla P1,\ell$. By assumption, as pointed out immediately after \Cref{defin:rw}, if we denote by $\pi$ the projection on the first two coordinates, each  $\pi(\star\Lambda_P)$ contains a Ramsey pair. It follows that  $\set{(\alpha,\beta)\mid (\alpha,\beta)\models\mathrm{RW}}\cap \bigcap_P\pi(\star\Lambda_P)$ is the intersection of a family of closed sets with the FIP, hence contains a pair $(\alpha,\beta)$, which is by construction Ramsey.  Finally, by saturation and downward directedness, $\bigcap_P\pi(\star\Lambda_P) = \pi\left(\bigcap_P(\star\Lambda_P)\right)$, and for any $(\alpha,
\beta)$ in $\set{(\alpha,\beta)\mid (\alpha,\beta)\models\mathrm{RW}}\cap \pi\left(\bigcap_P(\star\Lambda_P)\right)$ we can find the required $\freccia{\gamma_i}$.

$\ref{point:RWchargen}\Rightarrow\ref{point:RWcharultra}$ Let $u=\tp(\alpha,\beta)$ and, for $i\leq n$, let $u_{i}=\tp(({\freccia{\gamma_{i}})_{1}})$. As $\pi_{1}(\tp(\alpha,\beta))=\tp({\alpha})$, and $\alpha\sim(\freccia{\gamma_{1}})_{1}$ by hypothesis, we have that $\pi_{1}(u)=u_{1}$.  Given $A\in u, B_{1}\in u_{1},\ldots,B_{n}\in u_{n}$, define
\[I\coloneqq\{(a,b)\in A\mid a,b\in B_1\wedge \exists \freccia{z_{1}}\in B_{1},\ldots,\exists\freccia{z_{n}}\in B_{n}\ \varphi\left(a,b,\freccia{z_{1}},\ldots, \freccia{z_{n}}\right)\}.\]
Because $\varphi\left(\alpha,\beta,\gamma_{1},\ldots, \gamma_{m}\right)$ holds we have that $(\alpha,\beta)\in \star{I}$, 
hence $I\in u$. As $u\in\mathrm{RW}$, there exists an infinite set $H\subseteq\N$ such that $[H]^{2}\subseteq I$, which gives us the conclusion.

$\ref{point:RWcharultra}\Rightarrow\ref{point:RWchardef}$ Let the partition $\N=C_{1}\cup\ldots\cup C_{r}$ be given. For $j\leq n$, let $i_{j}$ be such that $C_{i_{j}}\in u_j$. This suffices to prove the conclusion. \end{proof}

 \Cref{teo: nonstandard first order properties RW} reduces Ramsey partition regularity problems to existence problems for Ramsey pairs with certain combinatorial properties. We are going to exploit this idea, look at formulas of the form $\bigwedge_{i\leq n} f_{i}(x,y)=z_{i}$, and study their Ramsey partition regularity in $(x,y)\mid\bla z1,n$.
By \Cref{teo: nonstandard first order properties RW}, this can be rephrased as follows.

\begin{co}\label{teo:RWequiv}
    Let $f_{1},\ldots,f_{n}\from \N^2\to\N$.
    The following are equivalent:
    \begin{enumerate}
    \item  The formula  $\bigwedge_{i\leq n} f_{i}(x,y)=z_{i}$  is Ramsey PR in $(x,y)\mid\bla z1,n$.
        \item There is a Ramsey pair $(\alpha, \beta)$ such that $f_{1}(\alpha, \beta)\sim\ldots\sim f_{n}(\alpha, \beta)$.
        \item \label{point:rwnoet} There is $w\in \mathrm{RW}$ such that $f_{1}(w)=\ldots=f_{n}(w)$.
    \end{enumerate}
  \end{co}
Note that,  if one of the $f_i$ is the projection on the first (or second) coordinate, then we are actually studying the
Ramsey partition regularity in $(x,y),\bla z1,n$.
One might wonder whether, in condition~\ref{point:rwnoet}, $\mathrm{RW}$ could be strengthened to $\mathrm{TS}$. This is not the case, as we will see in \Cref{rem:noETdifference}.

\section{Pairwise sums and products: a case study}\label{sums and products}\label{sec:casestudy}
In the language introduced in this paper, the \nameref{thm:pwsp} says that the formula $\varphi(x,y,z,t)\coloneqq (x+y=z)\wedge (x\cdot y = t)$ is not Ramsey PR in $(x,y)\mid z,t$. By \Cref{teo:RWequiv} this is, in disguise, a result about Ramsey pairs, namely the following. 

\begin{thm}\label{thm: no sums and products RW}
   Let $(\alpha, \beta)\models \mathrm{RW}$. Then $\alpha+\beta\nsim\alpha\beta$.
\end{thm}

We show how our techniques can be used to give a  short nonstandard proof of this, and in fact prove a much stronger result. Let us begin by fixing some notation.

  \begin{defin}
    Let $p,m\in\N\setminus\set1$, with $p$ prime. We denote by $v_p$ the $p$-adic valuation, and we let $\ell_m(x)\coloneqq \floor{\log_m(x)}$, both regarded as functions $\mathbb N\to \mathbb N\cup \set 0$.  We denote by $\smod_p\from \N\rightarrow\{1,\ldots,p-1\}$  the function that maps $x$ to $x/p^{v_{p}(x)}\mod p$.

When $p,m$ are clear from context, we write $\ell,v,\smod$ for brevity.
  \end{defin}
  Note that, if $x$ is written in base $p$, and we count digit positions right to left\footnote{So ``first'' means ``least significant''.} and starting from $0$, then $v_p(x)$ is the position of the first nonzero digit, $\smod_p(x)$ is that digit, and $\ell_p(x)$ is the position of the last nonzero digit.

  \begin{lemma}\label{rem:exsecondacifra}
Let $n,m\in  \mathbb Z\setminus \set 0$ with $n+m\ne 0$ and  $p>\abs n+\abs m$ a standard prime. If  $\alpha\sim\beta$ and $v_p(\alpha)\le v_p(\beta)$, then $v_p(n\alpha+m\beta)=v_p(\alpha)$.
  \end{lemma}
  \begin{proof}
Write $v\coloneqq v_p$ and $\smod\coloneqq\smod_p$. If $v(\alpha)<v(\beta)$ we conclude by basic valuation theory and the fact that $p\nmid n$, so we may assume that $v(\alpha)=v(\beta)$.
Observe that $v(n\alpha+m\beta)\ge v(\alpha)$. Moreover, as $\smod$ has finite image and $\alpha\sim \beta$, we have $\smod(\alpha)=\smod(\beta)$, hence the digit in position $v(\alpha)$ of $n\alpha+m\beta$ is $(n+m)\smod(\alpha)\mod p$, which is not $0$ by the assumptions on $p$ and $n+m$.
  \end{proof}
The main result of this section is the following:

\begin{teo}\label{lemma: RW then pairs are not equivalent}
    Let $(\alpha, \beta)\models \mathrm{RW}$ and $f, g\from \N\to \mathbb Z$. 
    Then, provided that $f(\alpha),g(\alpha)\notin \mathbb Z$, \[(f(\alpha)+f(\beta), \:g(\alpha)+g(\beta))\nsim(f(\alpha),\: g(\beta)).\]
    More generally, for every $k_1, k_2, k_3, k_4\in \Z$, \[(f(\alpha)+f(\beta)+k_1,\; g(\alpha)+g(\beta)+k_2)\nsim(f(\alpha)+k_3, \:g(\beta)+k_4).\]
\end{teo}

\begin{proof}
  Up to replacing $f$ by $-f$, and similarly for $g$, we may assume that $f(\alpha),g(\alpha)>\mathbb N$. By \Cref{fact:simprop}(\ref{point:simpushf}), by replacing $f$ by $f+k_3-k_1$ and $g$ by $g+k_4-k_2$ we may assume that $k_1=k_3$ and $k_2=k_4$. By applying  $(\id-k_1,\id-k_2)$ we then reduce to the case $k_1, k_2, k_3, k_4=0$.

  Suppose that $(f(\alpha)+f(\beta),\: g(\alpha)+g(\beta))\sim(f(\alpha),\; g(\beta))$. Fix a prime $p>2$ and, for brevity, let $v\coloneqq v_p$,  and similarly for $\ell$ and $\smod$.     

  Observe that $v(f(\alpha))\ne v(f(\beta))$; in fact, by assumption  $f(\alpha),f(\beta)\ne 0$ and, if they had different valuation, because $p>2$ and $f(\alpha)\sim f(\beta)$, we would have $\smod(f(\alpha)+f(\beta))\equiv 2\smod(f(\alpha))\pmod p$, contradicting $f(\alpha)+f(\beta)\sim f(\alpha)$.

Thus, since $(\alpha, \beta)\models\mathrm{RW}$, by \Cref{co:RWbasics}(\ref{point:RW7}) we have $v(f(\beta))>\dcl(\alpha)$. Similarly, $v(g(\beta))>\dcl(\alpha)$, which implies $v(g(\alpha)+g(\beta))=v(g(\alpha))$. Now observe that \[\ell(f(\alpha)+f(\beta))\ge v(f(\beta))>\ell(g(\alpha))+1>v(g(\alpha))=v(g(\alpha)+g(\beta)).\]

As the pairs are equivalent, it follows that $\ell(f(\alpha))>v(g(\beta))$, against $v(g(\beta))>\dcl(\alpha)$.
\end{proof}

\Cref{thm: no sums and products RW} can now be deduced from \Cref{lemma: RW then pairs are not equivalent} as follows.

\begin{proof}[Proof of  \Cref{thm: no sums and products RW}]
    Suppose by contradiction this is the case, fix a sufficiently large standard prime $p$ and write $v\coloneqq v_p$ and similarly for $\ell$ and $\smod$.
    By applying the function $(v, \ell)$ we get $(v(\alpha\beta),\; \ell(\alpha\beta))\sim (v(\alpha+\beta),\; \ell(\alpha+\beta))$.    We have \begin{itemize}
        \item $v(\alpha\beta)=v(\alpha)+v(\beta)$.
        \item $\ell(\alpha\beta)=\ell(\alpha)+\ell(\beta)+\varepsilon$ with $\varepsilon\in \{0, 1\}$.
        \item $v(\alpha+\beta)=v(\alpha)$: in fact, $v(\alpha)\le v(\beta)$ by \Cref{co:RWbasics}(\ref{point:RW7}), so we may apply \Cref{rem:exsecondacifra}.
        \item $\ell(\alpha+\beta)=\ell(\beta)+\varepsilon'$ with $\varepsilon'\in\{0, 1\}$.
    \end{itemize}
    Thus $(v(\alpha)+v(\beta),\;\ell(\alpha)+\ell(\beta)+\varepsilon)\sim (v(\alpha),\;\ell(\beta)+\varepsilon')$, and since $\ell(\alpha)>\mathbb N$ \Cref{lemma: RW then pairs are not equivalent} implies that $v(\alpha)\in \mathbb N\cup \set 0$.

The ultrafilter $\tp(\alpha)$ determines, for each power $p^k$, a remainder class modulo $p^k$. In other words, it determines the $p$-adic class $a\in \mathbb Z_p$ of $\alpha$. Therefore, since $\alpha\beta\sim \alpha+\beta$ we obtain $a^2=2a$ in $\mathbb Z_p$. Because we proved above that $v(\alpha)\in \mathbb N\cup \set 0$, it follows that $a=2$.

    The pair $(\gamma, \delta)\coloneqq (\alpha-2,\beta-2)$ is Ramsey by \Cref{co:RWbasics}(\ref{point:RWfrw}), and $\alpha\beta\sim \alpha+\beta$ yields 
    \begin{equation}
      \label{eq:3}
      \gamma\delta+2(\gamma+\delta)\sim \gamma+\delta.
    \end{equation}
    As $\alpha\notin \mathbb N$ we have $\gamma \ne 0$, and since the class of $\gamma$ in $\mathbb Z_p$ is $0$ it follows that $v(\gamma)>\mathbb N$. By \Cref{rem:exsecondacifra} and the fact that $p>2$,
    \[
      v(\gamma\delta)> v(\gamma)=v(\gamma+\delta)=v(2(\gamma+\delta)).
    \]
Therefore, $\smod(\gamma\delta+2(\gamma+\delta))=\smod(2(\gamma+\delta))$, contradicting~\eqref{eq:3}.
\end{proof}

Let us observe that, with the same idea as in the proof of  \Cref{lemma: RW then pairs are not equivalent}, one proves the following.

\begin{rem}\label{lemma:mariaclara}
Let $(\alpha,\beta)\models\mathrm{RW}$. For every $f,g,h,k$, if there are $\phi,\psi$ such that $\phi(g(\beta))\ne \phi(g(\alpha))$ and $\phi(k(\alpha,\beta))\le \psi(h(\alpha,\beta))$, then $(h(\alpha,\beta), k(\alpha,\beta))\nsim (f(\alpha), g(\beta))$.  Otherwise we get $\phi(g(\beta))>\dcl(\alpha)$, and in particular $\phi(g(\beta))>\psi(f(\alpha))$, against $\phi(k(\alpha, \beta))\leq \psi(h(\alpha, \beta))$.
\end{rem}

\Cref{thm: no sums and products RW} says that the formula
\[
  \varphi(x,y,z,t)\coloneqq (x+y=z)\wedge (x\cdot y=t) \wedge x\ne y.
\]
is not Ramsey PR in $(x,y)\mid z,t$. We now look at its other block and Ramsey partition regularity properties.

As known, the partition regularity of $\varphi(x,y,z,t)$ is one of the major open problems in this area of combinatorics. It has been recently solved over $\Q$~\cite{bowenMonochromaticProductsSums2024}; see also \cite{Alweiss2024Monochromatic} for an extension to finite arbitrarily large sums and products. On $\N$, the most general result we are aware of is the following theorem proven by J.~Moreira in~\cite[Corollary~1.5]{moreiraMonochromaticSumsProducts2017}, which we state here using the language of block partition regularity.

\begin{thm}[Moreira]\label{moreira} $\varphi(x,y,z,t)$ is BPR in $x,z,t\mid y$.\end{thm}

Moreira's Theorem, together with some basic considerations, solves most block partition regularity problems for $\varphi(x,y,z,t)$. We list the problems and their status in \Cref{tab:BlockPartitionRegularityOfVarphiXYZT}. We write ``Never considered'' when we are not aware of any work on the problem and, for compactness, we use the following observation.

\begin{rem}\label{rem:redundantBPR}
  We may omit redundant choices of blocks, arising from the symmetry in $x,y$ of the formula $\phi(x,y,z,t)$, or from the fact that variables inside a block can be permuted.
\end{rem}

\begin{table*}
  \centering
  \begin{tabular}{c|c|c}
    \textbf{Blocks} & \textbf{Is it BPR?} & \textbf{Reason}\\
    $x,y,z,t$ & Unknown & Open problem since the 1970's\\
    $x,y,z\mid t$ & Yes & Additive Schur's Theorem\\
    $x,y,t\mid z$ & Yes & Multiplicative Schur's Theorem\\
    $x,z,t\mid y$ & Yes & Moreira's Theorem\\
    $x,y\mid z,t$ & Unknown & Never considered\\			
    $x,z\mid y,t$ & Unknown & Never considered\\
    $x,t\mid y,z$ & Unknown & Never considered\\
    $x,y\mid z\mid t$ & Yes & Trivial\\
    $x,z\mid y\mid t$ & Yes & Follows from $x,z,t\mid y$\\
    $x,t\mid y\mid z$ & Yes & Follows from $x,z,t\mid y$\\
    $z,t\mid x\mid y$ & Yes & Follows from $x,z,t\mid y$\\
    $x\mid y\mid z\mid t$ & Yes & Trivial
  \end{tabular}
  \caption{Block Partition Regularity of $\varphi(x,y,z,t)$.}
  \label{tab:BlockPartitionRegularityOfVarphiXYZT}
\end{table*}

The picture is much richer when it comes to the different kinds of Ramsey partition regularity one might ask for $\varphi(x,y,z,t)$. First, we summarise what happens when variables in the ``Ramsey part'' are in a block of their own. Three such cases need some simple computations to be carried out, which we do in the following lemma.

\begin{lemma}\label{lemma tabellina} The formula $\varphi(x,y,z,t)$ is Ramsey PR in $(x,t)\mid y\mid z$ and $(x,z)\mid y\mid t$, but it is not Ramsey PR in $(z,t)\mid x\mid y$.\end{lemma}

\begin{proof} For the Ramsey partition regularity in $(x,t)\mid y\mid z$, fix $(\alpha,\beta)\models\mathrm{RW}$ and let $(x,t)=(2^{\alpha},2^{\beta})$, while $y=2^{\beta-\alpha}$ and $z=2^{\alpha}+2^{\beta-\alpha}$.

For the Ramsey partition regularity in $(x,z)\mid y\mid t$, fix any $(x,z)=(\alpha,\beta)\models\mathrm{RW}$ and let $y=\beta-\alpha$ and $t=\alpha\cdot (\beta-\alpha)$.

For the non Ramsey partition regularity in $(z,t)\mid x\mid y$, assume by contrast that there are $(\alpha,\beta)\models\mathrm{RW}$ and $\gamma,\delta\in\star\N$ such that $\alpha=\gamma+\delta$ and $\beta=\gamma\delta$. Then $\alpha^{2}=(\gamma+\delta)^{2}=\gamma^{2}+2\gamma\delta+\delta^{2}>\gamma\delta=\beta$, against  \Cref{co:RWbasics}(\ref{point:RW8}). \end{proof}

\Cref{tab:RamseyPartitionRegularityOfVarphiXYZT} lists all the cases where the ``Ramsey part'' is in a block of its own, except those excluded by \Cref{rem:redundantBPR} or by the following \Cref{rem:omittedcases}.

\begin{rem}\label{rem:omittedcases}
We may omit cases containing $(z,x)$, $(t,x)$ or $(t,z)$ which, if there were a pair $(\alpha,\beta)\models \mathrm{RW}$ witnessing Ramsey PR, would yield $\alpha+\beta<\alpha$, $\alpha\cdot\beta<\alpha$, and $\alpha\cdot\beta<\alpha+\beta$ respectively.
\end{rem}

\begin{table*}
	\centering
		\begin{tabular}{c|c|c}
			\textbf{Blocks} & \textbf{Is it Ramsey PR?} & \textbf{Reason}\\
			$(x,y)\mid z\mid t$ & Yes & Trivial\\
			$(x,t)\mid y\mid z$ & Yes & \Cref{lemma tabellina}\\
			$(x,z)\mid y\mid t$ & Yes & \Cref{lemma tabellina}\\
			$(z,t)\mid x\mid y$ & No & \Cref{lemma tabellina}\\	
			$(x,y)\mid z,t$ & No & \Cref{thm: no sums and products RW}\\
			$(x,z)\mid y,t$ & Unknown & Never considered\\
			$(x,t)\mid y,z$ & Unknown & Never considered\\
			$(z,t)\mid x,y$ & No &  Follows from $(z,t)\mid x\mid y$
		\end{tabular}
	\caption{Ramsey partition regularity of $\varphi(x,y,z,t)$, case with separate Ramsey block.}
	\label{tab:RamseyPartitionRegularityOfVarphiXYZT}
\end{table*}

Finally, we are left to consider the cases where  variables in the ``Ramsey part'' share a block with other variables. Some such cases can be easily proven by simple computations.

\begin{lemma}\label{lemma tabellina 2} The formula $\varphi(x,y,z,t)$ is Ramsey PR in $(x,z),y\mid t$ and $(x,t),y\mid z$.\end{lemma}

\begin{proof} Let $(\alpha,\beta,\gamma)$ be a tensor triple of infinite equivalent elements.

For the $(x,z),y\mid t$ case we let $x=\beta-\alpha$, $z=\gamma-\alpha$, $y=\gamma-\beta$ and $t=\left(\beta-\alpha\right)\left(\gamma-\beta\right)$.

For the $(x,t),y\mid z$ case we additionally require that  $\alpha$ is divisible by all $n\in\N$, and we let $x=\beta/\alpha, t=\gamma/\alpha, y=\gamma/\beta$ and $z=\left(\beta/\alpha\right)+\left(\gamma/\beta\right)$. \end{proof}

We use \Cref{lemma tabellina 2} to list all remaining cases in \Cref{tab:RamseyPartitionRegularityOfVarphiXYZT2}, again with the omissions arising from \Cref{rem:redundantBPR,rem:omittedcases}.
\begin{table*}[b]
  \centering
  \begin{tabular}{c|c|c}
    \textbf{Blocks} & \textbf{Is it Ramsey PR?} & \textbf{Reason}\\
    $(x,y),z\mid t$ & Yes & Use additive idempotents\\
    $(x,y),t\mid z$ & Yes & Use multiplicative idempotents\\		
    $(x,z),y\mid t$ & Yes & \Cref{lemma tabellina 2}\\
    $(x,z),t\mid y$ & Unknown & Never considered\\		
    $(x,t),y\mid z$ & Yes & \Cref{lemma tabellina 2}\\
    $(x,t),z\mid y$ & Unknown & Never considered\\
    $(z,t),x\mid y$ & No & Follows from $(z,t)\mid x\mid y$\\
    $(x,y),z,t$ & No & Follows from $(x,y)\mid z,t$\\
    $(x,z),y,t$ & Unknown & Stronger than $x,y,z,t$\\
    $(x,t),y,z$ & Unknown & Stronger than $x,y,z,t$\\
    $(z,t),x,y$ & No &  Follows from $(z,t)\mid x\mid y$
  \end{tabular}
  \caption{Ramsey partition regularity of $\varphi(x,y,z,t)$, remaining cases.}
  \label{tab:RamseyPartitionRegularityOfVarphiXYZT2}
\end{table*}

\section{Ramsey PR of equations}\label{pr poly}

In this section we apply the characterisations obtained in \Cref{teo:RWequiv} to  study the Ramsey partition regularity of several equations. We begin with polynomial ones. The two-variable case is trivial, for the following reason.

\begin{fact}\label{fact:twovar}
  If $\alpha\sim \beta$ are infinite, then $P(\alpha,\beta)=0$  if and only if  $\alpha=\beta$ and $P(x,y)$ is a multiple of $(x-y)$.
\end{fact}

\begin{proof}
   By~\cite[Corollary 3.2]{ArrudaLBLimiting24} and \Cref{fact:fundprop}.
\end{proof}

\begin{pr}\label{rem:no2varRPR}
No two-variable polynomial equation is Ramsey PR. 
\end{pr}

\begin{proof}
We use \Cref{fact:twovar}. Since every Ramsey pair $(\alpha,\beta)$ satisfies $\alpha<\beta$ and $\alpha\sim \beta$ (\Cref{co:RWbasics}, points~\ref{point:RW7} and~\ref{point:RW5} respectively), no such pair can satisfy the formula $P(x,y)=0$. 
\end{proof}

We start our discussion by fixing some notations and proving some simple, yet useful, general result.

\begin{defn} For $P(x)\in\Z[x]$, we let $\deg(P)$ be its degree and $c(P)$ be its leading coefficient.\end{defn}

When evaluating polynomials at infinite points, it is natural to consider the following notion:

\begin{defn} Let $\alpha,\beta\in\star{\N}$. We say that $\alpha,\beta$ are in the same \emph{archimedean class}, and write $\alpha\asymp \beta$, if there exists $n\in\N$ such that $(1/n)\leq (\alpha/\beta)\leq n$. We write $\alpha\ll \beta$ if the archimedean class of $\alpha$ is strictly smaller than that of $\beta$, that is, if for all $n\in \mathbb N$ we have $n\cdot \alpha <\beta$.
\end{defn}

A related way to think about $\asymp$ is the following: let us recall that a hyperreal $\rho\in\star{\mathbb{R}}$ is called \emph{finite} if there exists $n\in\N$ such $|\rho|\leq n$. In this case, $\rho$ admits a \emph{standard part}, denoted by $\st(\rho)$ and defined as the unique standard real $r$ whose distance from $\rho$ is \emph{infinitesimal}, i.e.\ smaller than $1/n$ for all $n\in\N$. Saying $\alpha\asymp\beta$ means saying that $\alpha/\beta$ is finite and its standard part is not $0$, while saying that $\alpha\ll \beta$ means that $\st(\alpha/\beta)=0$.

By \Cref{fact:simprop}(\ref{point:fideq}), if $\alpha\sim\beta$ and $q\in\Q$, then $\alpha=q\beta$ forces $q=1$. The first part of next lemma is a simple generalisation of this fact where $=$ is replaced by $\asymp$.

\begin{lemma}\label{tecnico lemma 2} Let $\alpha\sim\beta$ be infinite and  $P,Q\in\Z[x]$  be polynomials of degrees   $\deg(P)=n$ and $\deg(Q)=m$. If $n,m\geq 1$, then:
\begin{enumerate}
\item\label{point:teclem1} if $\alpha\asymp\beta$ then $\st\left({\alpha}/{\beta}\right)=1$;
\item\label{point:teclem2} if $P(\alpha)\asymp Q(\beta)$ then $n=m$;
\item if $P(\alpha)\sim Q(\beta)$ then $P=Q$.
\end{enumerate}
\end{lemma}

\begin{proof}Fix an arbitrary finite prime $p$ and let $\ell\coloneqq\ell_p$.
  \begin{enumerate}[wide=0pt, labelwidth=!, labelindent=0pt]
  \item  Write $\alpha,\beta$ in base $p$. In the notations of our previous section, $\ell(\alpha)\sim \ell(\beta)$. Hence either $\ell(\alpha)=\ell(\beta)$ or the value $|\ell(\alpha)-\ell(\beta)|$ is infinite, and the latter is impossible as in this case $\alpha/\beta$ would either be infinite or infinitesimal, against our hypothesis. Let $\lambda=\ell(\alpha)=\ell(\beta)$. As $\alpha\sim\beta$, for all finite $i$ the coefficients of $p^{\lambda-i}$ in the base $p$ expansions of $\alpha$ and $\beta$ coincide; then, by overspill, there exists an infinite $\tau$ such that for every $i\leq \tau$ the $(\lambda-i)$-th coefficient $c_{\lambda-i}$ in the base $p$-expansions of $\alpha$ and $\beta$ is the same. Let $\gamma=\sum_{i=0}^{\tau}c_{\lambda-i}p^{\lambda-i}$. Then:
    \begin{itemize}
    \item $\alpha=\gamma +\varepsilon$, with $\varepsilon\ll \gamma$;
    \item $\beta=\gamma+\varepsilon^{\prime}$, with $\varepsilon^{\prime}\ll \gamma$.
    \end{itemize}
    Hence $\alpha/\beta=(\gamma+\varepsilon)/(\gamma+\varepsilon^{\prime})$
has standard part $1$.

  \item  First of all, we observe that 
\[\alpha^{n}\asymp P(\alpha)\asymp Q(\beta)\asymp \beta^{m},\]
so $\alpha^{n}\asymp \beta^{m}$.  Therefore,  $\ell(\alpha^n)$ is at finite distance from $\ell(\beta^m)$, so there is a finite $k\in \mathbb Z$ such that  $n\ell(\alpha)+k=m\ell(\beta)$. As $\ell(\alpha)\sim\ell(\beta)$, if this happens then the equation $nx-my=k$ is PR. By Rado's Theorem, this can happen only when $n=m$ (and $k=0$).

\item The assumptions imply $P(\alpha)\sim Q(\alpha)$. As polynomials are bounded-to-one, by \Cref{pr:bddto1} we must have $P(\alpha)=Q(\alpha)$. So $P(\alpha)-Q(\alpha)=0$, and since $\alpha$ is infinite we must have $P=Q$.\qedhere
\end{enumerate}
\end{proof}

\Cref{tecnico lemma 2} allows some simple, yet interesting, observations. For instance, we may combine it  with Rado's Theorem and van der Waerden's Theorem to obtain the following propositions.

\begin{pr}\label{pr:asymptwitness}
Let $a,b,c\in\N$.
\begin{enumerate}
\item If $a=c\ne b$, then for all finite partitions $\N=C_{1}\cup\ldots\cup C_{r}$ and all $N\in\N$, there exist $i\leq r$ and $x,y,z\in C_{i}$ with $x>Ny$ and such that $ax+by=cz$. Moreover, the analogous result with $x>Ny$ replaced by $y>Nx$ is false.
\item If $a+b=c$  then for all finite partitions $\N=C_{1}\cup\ldots\cup C_{r}$ and all $\epsilon>0$ there are $i\le r$ and $x,y,z\in C_{i}$ with $\abs{(x/y)-1},\abs{(x/z)-1}<\epsilon$ and such that $ax+by=cz$.
Moreover, for all $N>1$, there is a finite partition $\N=C_{1}\cup\ldots\cup C_{r}$ such that for all $i\le r$ and all $x,y,z\in C_i$ with $x> Ny$ we have $ax+by\ne cz$. Similarly with $x,y$ replaced, in $x>Ny$, by any two distinct elements of $\set{x,y,z}$.
\end{enumerate}
\end{pr}

\begin{proof}
   By Rado's Theorem,  the equation $ax+by=cz$ is PR if and only if $a=c$, or $b=c$, or $a+b=c$. In nonstandard terms, this equation is PR if and only if there exist $\alpha\sim\beta\sim\gamma$ such that $a\alpha+b\beta=c\gamma$. Fix such $\alpha,\beta,\gamma$.
 \begin{enumerate}[wide=0pt]
 \item It suffices to show that $\beta\ll\alpha$. If $\alpha\asymp \beta$ then $a\alpha+b\beta\asymp (a+b)\alpha$, so $c\gamma\asymp (a+b)\alpha$, hence by \Cref{tecnico lemma 2}(\ref{point:teclem1}) we should have $a+b=c$, against the assumption. If $\alpha\ll \beta$, it similarly follows that $b=c$.  
 \item It suffices to show that $\st(\alpha/\beta)=\st(\beta/\gamma)=1$, and by \Cref{tecnico lemma 2}(\ref{point:teclem1}) it is enough to prove $\alpha\asymp \beta\asymp \gamma$. This follows easily from the equation $a\alpha+b\beta=(a+b)\gamma$ and again  \Cref{tecnico lemma 2}(\ref{point:teclem1}).\qedhere
 \end{enumerate}
\end{proof}

\begin{pr}
 For every $N\in\N$, every $k>1$, and every finite partition $\N=C_{1}\cup\ldots\cup C_{r}$, there exist $i\leq r$ and $x,y\in \N$ such that $x>Ny$ and for all $0\le j\leq k$ we have  $x+jy\in C_{i}$. Moreover, the analogous result with $x>Ny$ replaced by $y>Nx$ is false.  
\end{pr}

\begin{proof}
By the nonstandard characterisation of van der Waerden's Theorem there exists $\alpha,\beta$ such that for all $n\in\N$ we have $\alpha\sim \alpha+n\beta$. Arguing as in the proof of \Cref{pr:asymptwitness} we deduce that $\beta\ll \alpha$. From this, the conclusion follows easily.
\end{proof}

The last technical result that we will need is the following:

\begin{lemma}\label{tecnico lemma} Let $P(x),Q(x)\in \Z[x]$ be such that $\deg(P)=\deg(Q)=d>1$. Assume that $c(P),c(Q)>0$ and let $c\coloneqq {c(P)}/{c(Q)}$. For all $\beta\in\star{\N}\setminus\N$ there exists $n\in\N$ such that $Q(\sqrt[d]{c}\beta-n)\leq P(\beta) \leq Q(\sqrt[d]{c}\beta+n)$. \end{lemma}

\begin{proof} Let $c'(P), c'(Q)$ be the coefficients of $x^{d-1}$ in $P,Q$ respectively. Observe that, as $\beta$ is infinite, our hypotheses on $P,Q$ give us that $Q(\sqrt[d]{c}x-n), P(x), Q(\sqrt[d]{c}x+n)$ all have the same leading monomial, hence 
  the inequality in the conclusion holds if and only if the same inequality holds for the coefficients of $x^{d-1}$ in $Q(\sqrt[d]{c}x-n), P(x), Q(\sqrt[d]{c}x+n)$ respectively. This coefficient is $(c'(Q)-dn \cdot c(Q))c^\frac{d-1}d$ in $Q(\sqrt[d]{c}x-n)$ and it is $(c'(Q)+dn\cdot c(Q))c^\frac{d-1}d$ in $Q(\sqrt[d]{c}x+n)$. To conclude, it suffices to find some $n$ such that 
  \[
    (c'(Q)-dn\cdot c(Q))c^\frac{d-1}d\leq c'(P)\leq (c'(Q)+dn\cdot c(Q))c^\frac{d-1}d.
  \]
Because $d\cdot c(Q)$ and $c$ are positive, it is enough to take $n$  sufficiently large.
\end{proof}

We are now ready to discuss the Ramsey partition regularity of polynomial equations in three variables. Our first result is very general, as it involves arbitrary functions in $x$.

\begin{teo}\label{teo: poly bounds RW} Let $P(x),Q(x)\in \Z[x]$ be nonzero polynomials with $P(0)=Q(0)=0$. Let $f\from \mathbb N\to \mathbb Z$  and let $\varphi(x,y,z)\coloneqq  (f(x)+P(y)=Q(z))\land (f(x)\ne 0)$.
If $\varphi(x,y,z)$ is Ramsey PR in $(x,y),z$ then $P=Q$ and $\deg(P)=1$.\end{teo}

\begin{proof} Let $d_{P}=\deg(P), d_{Q}=\deg(Q)$. By \Cref{teo: nonstandard first order properties RW}, as $\varphi(x,y,z)$ is Ramsey PR in $(x,y),z$ there exist $(\alpha,\beta)\models \mathrm{RW}$ and $\gamma\sim\alpha$ such that $f(\alpha)+P(\beta)=Q(\gamma)$ and $f(\alpha)\ne 0$; note that, by  \Cref{co:RWbasics}(\ref{point:RW8}), $f(\alpha)\ll\beta$. Hence 
\[\beta^{d_{P}}\asymp f(\alpha)+P(\beta) =Q(\gamma)\asymp\gamma^{d_{Q}}.\]
By \Cref{tecnico lemma 2}(\ref{point:teclem2}) we deduce that $d_{Q}=d_{P}$.

Observe that the sign of $f(\alpha)+P(\beta)$ coincides with that of $c(P)$, and that of $Q(\gamma)$ coincides with that of $c(Q)$; as $f(\alpha)+P(\beta)= Q(\gamma)$, necessarily $c\coloneqq{c(P)}/{c(Q)}>0$. Without loss of generality, both $c(P)$ and $c(Q)$ are positive.

Let $d=d_{Q}=d_{P}$. Assume that $d>1$. By  \Cref{tecnico lemma} applied twice, once to the polynomials $P(x)-x$ and $Q(x)$ and once to the polynomials $P(x)+x$ and $Q(x)$, there is $n\in\N$ such that 
\[Q(\sqrt[d]{c}\beta-n)\leq P(\beta)-\beta \leq f(\alpha)+P(\beta)\leq P(\beta)+\beta \leq Q(\sqrt[d]{c}\beta+n).\]
So $Q(\sqrt[d]{c}\beta-n)\leq Q(\gamma)\leq Q(\sqrt[d]{c}\beta+n)$, and as $Q$ is increasing on $\star{\N}\setminus\N$, this entails that $\gamma\in [\sqrt[d]{c}\beta-n,\sqrt[d]{c}\beta+n]$. So there exists $i\in [-n,n]$ such that $\gamma=\floor{\sqrt[d]{c}\beta}+i$, in particular, $\floor{\sqrt[d]{c}\beta}+i\sim\beta$, hence $\floor{\sqrt[d]{c}\beta}+i=\beta$. As $\beta$ is infinite, \Cref{fact:simprop}(\ref{point:fideq}) gives $\sqrt[d]{c}=1$ and $i=0$, hence $\gamma=\beta$. But then $f(\alpha)=(Q-P)(\beta)$, which forces $Q-P$ to be constant by \Cref{co:RWbasics}(\ref{point:RWfintoone}).  Since $P(0)=0=Q(0)$, this forces $Q=P$, which gives $f(\alpha)=0$, against the assumptions. Hence $d=1$.

 So we are in the following case:
\[f(\alpha)+c(P)\beta=c(Q)\gamma.\]
As $c(Q)\gamma\sim c(Q)\beta$ and  $f(\alpha)\ll\beta$, we conclude by  \Cref{tecnico lemma 2}(\ref{point:teclem1}).\end{proof}

In the particular case where $f(\alpha)$ is a monomial in $\alpha$, we can say more:

\begin{cor}\label{cor: poly applications} Let $P(x),Q(x)\in \Z[x]$ be nonzero polynomials with $P(0)=Q(0)=0$. Let $a,n\in\N$, and let $\varphi(x,y,z)\coloneqq  ax^{n}+P(y)=Q(z)$. If $\varphi(x,y,z)$ is Ramsey PR in $(x,y),z$ then $\varphi(x,y,z)$ is an integer multiple of $x+y=z$ or of $-x+y=z$.
\end{cor}

\begin{proof} Straight from \Cref{teo: poly bounds RW}, and using the same notations, we get that, by our hypotheses, there is some $c\in\Z$ such that $\varphi(x,y,z)$ is $ax^{n}+cy=cz$. Again, by \Cref{teo: nonstandard first order properties RW}, as $\varphi(x,y,z)$ is Ramsey PR in $(x,y),z$ there exist $(\alpha,\beta)\models\mathrm{RW}$ and $\gamma\sim\alpha$ such that $a\alpha^{n}+c\beta=c\gamma$.

Fix a prime number $p$ much larger than $a,c$ (in particular, $v(a)=0=v(c)$), and write $\alpha,\beta,\gamma$ in base $p$. In the same notation of \Cref{sums and products}, as $(\alpha,\beta)\models\mathrm{RW}$, by \Cref{co:RWbasics}(\ref{point:RWfrw}) and \Cref{co:RWbasics}(\ref{point:RW8}) we have only two possible cases: $v(\alpha)=v(\beta)$ or $v(\alpha)\ll v(\beta)$.

\paragraph{Case $v(\alpha)=v(\beta)$.} We divide this case in the two subcases $n>1$ and $n=1$. 

If $n>1$,  we observe that $v(a\alpha^{n})=v(a)+nv(\alpha)=nv(\alpha)$. If $v(\alpha)\ne 0$ then $nv(\alpha)>v(\alpha)=v(\beta)$, so necessarily $v(\beta)=v(\gamma)$. The same also holds if $v(\alpha)=0$ because $\alpha\sim\beta\sim\gamma$. Let $d(t)$ be the coefficient of $p^{n v(t)}$ in the base $p$ expansion of $t$. The $n v(\gamma)$ position in the base $p$ expansions of both sides of the equation $a\alpha^{n}+c\beta=c\gamma$ gives us that $a\smod(\alpha)^{n}+d(c\beta)\equiv d(c\gamma)\pmod p$. As $d$ has finite image we have $d(c\beta)= d(c\gamma)$. This implies $a\smod(\alpha)^{n}\equiv 0\pmod p$, which cannot happen since $a,\smod(\alpha)$ are not divided by $p$.

If $n=1$, then $v(a\alpha^{n})=v(\alpha)=v(\beta)$. Using the same notations as above, if $v(\beta)=v(\gamma)$ then again the equality $a\alpha+c\beta=c\gamma$  gives us $(a+c)\smod(\alpha)\equiv c\smod(\alpha)\pmod p$, which is impossible. Hence $v(\beta)<v(\gamma)$, so it must be $ad+cd\equiv 0\pmod p$, namely $a\equiv -c\pmod p$. As $p$ is much larger than $a,c$, this tells us that $a=-c$.

\paragraph{Case $v(\alpha)\ll v(\beta)$.} In this case, $v(\gamma)=v(a\alpha^{n}+c\beta)=nv(\alpha)$ which, as $v(\gamma)\sim v(\alpha)$, forces $n=1$. It follows that that $a\smod(\alpha)\equiv c\smod(\gamma)\pmod p$, which forces $a\equiv c\pmod p$. Again, as $p$ is much larger than $a,c$, this tells us that $a=c$.\end{proof}

The above result tells us that several well-known PR equations of the form $ax^{n}+P(y)=Q(z)$ are not Ramsey PR in $(x,y),z$, for example all equations of the form $x-y=Q(z)$ for $\deg(Q)>1$ and $Q(0)=0$ (see \cite{BFMC96} for a proof of their partition regularity). Interestingly, such equations can be PR over $\mathbb Q$, as discussed in the introduction. The reason that our argument breaks over $\mathbb Q$ is simply due to the fact that $\mathbb Q$ contains numbers of arbitrarily small (negative) $p$-adic valuation.

As for $\pm x+y=z$ (over $\mathbb N$), we have the following.

 \begin{prop}\label{pr:ramseydiff} Both $x+y=z$ and $-x+y=z$ are Ramsey PR in $(x,y),z$.\end{prop}

\begin{proof} As always, by \Cref{teo: nonstandard first order properties RW} it suffices to find $(\alpha,\beta)\models\mathrm{RW}$ and $\gamma\sim\alpha$ such that $\alpha+\beta=\gamma$ or $-\alpha+\beta=\gamma$. The $x+y=z$ case is well-known: it suffices to take $\alpha$ such that $\tp(\alpha)$ is additively idempotent\footnote{The existence of such ultrafilters is also well-known, see e.g.~\cite[Theorem 2.5]{hindmanAlgebraStoneCechCompactification2011}.} and $\beta$ such that $(\alpha,\beta)\models\mathrm{TS}$, so that by letting $\gamma=\alpha+\beta$ we have that $\gamma\sim\alpha$.

  For $-x+y=z$, pick any infinite  $\eta_{1}\sim\eta_{2}\sim\eta_{3}$ with $(\eta_{1},\eta_{2},\eta_{3})$ tensor. Let $\alpha=\eta_{2}-\eta_{1}, \beta=\eta_{3}-\eta_{1},\gamma=\eta_{3}-\eta_{2}$. Then $(\alpha,\beta)\models\mathrm{RW}$ by \Cref{co:intfunrw} and $\alpha\sim\beta\sim\gamma$ as they are all images of equivalent pairs under the same standard function. Clearly, $\beta-\alpha=\gamma$.\end{proof}

\begin{co}\label{co:proddiv}
  The equations $xy=z$ and $y=zx$ are Ramsey PR in $(x,y),z$.
\end{co}

\begin{proof}
By \Cref{pr:ramseydiff} there is $(\alpha,\beta)\models\mathrm{RW}$ with $\pm \alpha+\beta\sim \alpha$. It then suffices to consider the pair $(2^\alpha, 2^\beta)$.
\end{proof}

The equation $-x+y=z$ is the first example we are aware of with the following property.

\begin{rem}\label{rem:noETdifference}
The formula $-x+y=z$ is Ramsey PR in $(x,y),z$, but there are no $(\alpha,\beta)\models\mathrm{TS}$  such that $-\alpha+\beta \sim \alpha$. Namely, the equation $ (-u)\oplus u=u$ has no solutions in $\beta\Z\setminus \mathbb Z$.
\end{rem}

\begin{proof} The fact that such $(\alpha,\beta)$ cannot exist is well-known, see e.g.~\cite[Corollary 13.19]{hindmanAlgebraStoneCechCompactification2011}, but here is a quick argument.
 If $(\alpha,\beta)$ are as above, by Puritz' Theorem we have that $v_3(\beta)\in \mathbb N$ (otherwise $\smod_3(\alpha)=\smod_3(-\alpha)=3-\smod_{3}(\alpha)$, which cannot happen). But then $v_3(\alpha)=v_3(\beta-\alpha)$, a contradiction because $\smod_3(\beta)=\smod_3(\alpha)$, hence $v_{3}(\beta-\alpha)>v_{3}(\alpha)$. 
\end{proof}

 \Cref{cor: poly applications} has several more interesting consequences. One regards the non-Ramsey partition regularity of arithmetic progressions, not even in the simplest case of arithmetic progressions of length $3$. In particular, there is no ``Ramsey version'' of van der Waerden's Theorem.

\begin{prop}\label{no infinite 3-AP} The formula
  \[
    \varphi(x,y,z)\coloneqq \text{``$x,y,z$ form, in some order, an arithmetic progression of length $3$''} 
  \]
is not Ramsey PR in $(x,y),z$. \end{prop}

\begin{proof} First of all, let us observe that $\varphi(x,y,z)$ is equivalent to
  \begin{equation}
    \label{eq:1}
    (x+y=2z)\vee (x+z=2y)\vee (y+z=2x),
  \end{equation}
  whence by \Cref{teo: nonstandard first order properties RW} to conclude it suffices to prove that there are no $\alpha\sim\beta\sim\gamma$ satisfying~\eqref{eq:1} and such that $(\alpha,\beta) \models\mathrm{RW}$.
But $\alpha+\beta=2\gamma$ cannot happen due to  \Cref{cor: poly applications} applied to $P(x)=x$ and $Q(x)=2x$. The other two cases are dealt with similarly, by using \Cref{cor: poly applications} respectively with $P(x)=-2x$ and $Q(x)=-x$ and with $P(x)=-x$ and $Q(x)=x$.\end{proof}

Another application of \Cref{cor: poly applications} regards ultrafilter equations. It is known that equations of the form $au\oplus bu=cu$, where $a,b,c\in\Z\setminus\{0\}$ and $u\in\beta\N$, can be solved only when $a=b=c$ (see~\cite[Theorem~13.18]{hindmanAlgebraStoneCechCompactification2011},  and see~\cite{malekiSolvingEquationsBetaN2000} for other results of this kind). By \Cref{cor: poly applications} we know that much more is true:

\begin{prop}\label{prop Maleki} Let $P(x),Q(x)\in \Z[x]$ be nonzero polynomials with no constant term. Let $a,n\in\N$. The only equations of the form\footnote{Here $u^n$ is the image of the ultrafilter $u$ under the polynomial function $x\mapsto x^n$. This is distinct from $\underbrace{u\odot\ldots\odot u}_{n\text{ times}}$.} $au^{n}\oplus P(u)=Q(u)$ that have a solution in $\beta\N$ are the multiples of $u\oplus u=u$.\end{prop}

\begin{proof} Left to right, if the equation $au^{n}\oplus P(u)=Q(u)$ holds, then for $(\alpha,\beta)\models u\otimes u$ one has $a\alpha^{n}+P(\beta)\sim Q(\alpha)$. As $\mathrm{TS}\subseteq \mathrm{RW}$,  \Cref{cor: poly applications} forces $\deg(P)=\deg(Q)=n=1$ and $a=c(P)=c(Q)$. By \Cref{rem:noETdifference}, $a=1$. Right to left, any additively idempotent ultrafilter $u$ solves the given equation. \end{proof}

Another application of \Cref{tecnico lemma 2} is the following generalisation of the \nameref{thm:pwsp}.

    \begin{pr}
      Let $n,m\in \mathbb Z$, $k,r,s\in \mathbb N$. Assume that $n+m\ne 0$. The following are equivalent.
      \begin{enumerate}
      \item\label{point:gensumprodTS} There is $(\alpha,\beta)\models \mathrm{TS}$ such that $n\alpha+m\beta\sim k\alpha^r\beta^s$.
      \item\label{point:gensumprodRW} There is $(\alpha,\beta)\models \mathrm{RW}$ such that $n\alpha+m\beta\sim k\alpha^r\beta^s$.
      \item\label{point:gensumprodrsnm} $r=s=1$, one of $n,m$ equals $0$, and the other one is positive.
    \end{enumerate}
    \end{pr}
    \begin{proof}
      $\ref{point:gensumprodTS}\Rightarrow\ref{point:gensumprodRW}$ follows from \Cref{thm:tensors witness Ramsey Theorem}.
      
For $\ref{point:gensumprodRW}\Rightarrow\ref{point:gensumprodrsnm}$, assume that $(\alpha,\beta)\models \mathrm{RW}$. Fix a sufficiently large standard prime $p$ and let $v,\ell,\smod$ be relative to this prime.    If  $n\alpha+m\beta\sim k\alpha^r\beta^s$ then, by applying $\ell$ to both sides, and possibly replacing $\ell(\alpha)$ by $\ell(\beta)$ on the left hand side in the case where $m=0$, we find $\epsilon_0, \epsilon_1\in \mathbb N\cup \set0$ such that $\ell(\beta)+\epsilon_0\sim r\ell(\alpha)+s\ell(\beta)+\epsilon_1$. By \Cref{tecnico lemma 2}(\ref{point:teclem1}) we obtain $s=1$.  If $n,m\ne 0$, by \Cref{rem:exsecondacifra} we obtain $v(n\alpha+m\beta)=v(\alpha)$, hence $v(\alpha)\sim r v(\alpha)+v(\beta)$. The same conclusion still holds if one of $n,m$ equals $0$, by an easy calculation. If $c=\smod(v(\alpha))$ this yields $c\equiv rc\pmod p$ or $c\equiv (r+1)c\pmod p$ depending on whether $v(v(\alpha))=v(v(\beta))$ or not. As $r$ is positive and much smaller than $p$, this implies $r=1$.

      We have now reduced to the case $n\alpha+m\beta\sim k\alpha\beta$.  If $n,m$ are both nonzero, an easy modification of the proof of \Cref{thm: no sums and products RW} shows that the required $(\alpha,\beta)$ cannot be found. As $k\alpha\beta>0$, if $n=0$ then we must have $m>0$, and symmetrically with the roles of $n$ and $m$ interchanged. 

For $\ref{point:gensumprodrsnm}\Rightarrow\ref{point:gensumprodTS}$, we assume that $n=0<m$, the other case being analogous. We need to find $(\alpha,\beta)\models \mathrm{TS}$ such that $m\beta\sim k\alpha\beta$. Let $\alpha'$ be such that $\tp(\alpha')$ is idempotent for $\odot$ and divisible by $k$, and let $\beta'$ be such that $(\alpha',\beta')\models \mathrm{TS}$. It then suffices to set $\alpha\coloneqq m\alpha'/k$ and $\beta\coloneqq m\beta'/k$.
    \end{proof}
    \begin{rem}\label{rem:remainingcases}
      If we allow $r,s\in \mathbb Z$, similar arguments as above yield that, assuming $n+m\ne 0$, if there is $(\alpha,\beta)\models \mathrm{RW}$ with $n\alpha+m\beta\sim k\alpha^r\beta^s$ then $(r,s)\in \set{(1,1),(0,1),(1,0),(-1,1)}$. The cases $(0,1)$ and $(1,0)$ are equivalent and yield $n\alpha+m\beta\sim k\beta$. By \Cref{cor: poly applications} this entails $m=k=\pm n$. We do not know what happens in the case $(r,s)=(-1,1)$, except for (trivial consequences of) \Cref{co:proddiv}.
    \end{rem}

Let us consider what happens when we relax the Ramsey partition regularity from $(x,y),z$ to $(x,y)\mid z$. 

\begin{prop}\label{pr: Moreira triples} Let $a,b,c,k_{1},k_{2},k_{3}\in\N$. The following are equivalent:
\begin{enumerate}
\item\label{point:mortrip1} $ax^{k_{1}}+by^{k_{2}}= cz^{k_{3}}$ is Ramsey PR in $(x,y)\mid z$;
\item\label{point:mortrip2} $k_{3}=1$.
\end{enumerate}\end{prop}

\begin{proof} $\ref{point:mortrip1}\Rightarrow \ref{point:mortrip2}$ Assume, by contrast, that $k_{3}>1$. By \Cref{teo: nonstandard first order properties RW} there exist $(\alpha,\beta)\models \mathrm{RW}$ and $\gamma\in\star{\N}$ such that $a\alpha^{k_{1}}+b\beta^{k_{2}}= c\gamma^{k_{3}}$. Let $f\from \N\rightarrow\N$ be defined as follows:
\[f(y)\coloneqq \min\{z\in\N\mid cz^{k_{3}}\geq by^{k_{2}}\}.\]
We claim that $f(\beta)=\gamma$. In fact, trivially $f(\beta)\leq \gamma$. If $\gamma\geq f(\beta)+1$ then $c\gamma^{k_{3}}\geq c(f(\beta)+1)^{k_{3}}$. Write $c(f(\beta)+1)^{k_{3}}=cf(\beta)^{k_{3}}+R(\beta)$, hence $a\alpha^{k_{1}}+b\beta^{k_{2}}\geq cf(\beta)^{k_{3}}+R(\beta)$. Observe that, as $k_{3}\geq 2$, we have that $R(\beta)$ is a polynomial in $\beta$ of degree at least $1$.  As $cf(\beta)^{k_{3}}\geq b\beta^{k_{2}}$ by definition of $f$, we deduce that $a\alpha^{k_{1}}\ge R(\beta)$, and this is in contrast with \Cref{co:RWbasics}(\ref{point:RWfintoone}).

$\ref{point:mortrip2}\Rightarrow \ref{point:mortrip1}$ We just have to take $(\alpha,\beta)\models \mathrm{RW}$ with $c\mid\alpha$ and  let $\gamma=\left( a\alpha^{k_{1}}+b\beta^{k_{2}}\right)/c$.\end{proof}

\begin{esem} Let us show two examples.
\begin{enumerate}
\item As recalled before, all equations of the form $-x+y=cz^{n}$ are PR, but the only ones that are Ramsey PR in $(x,y)\mid z$ are those of the form $-x+y=cz$, among which the only one that is Ramsey PR in $(x,y),z$ is $-x+y=z$.
\item Recall (\Cref{eg:pythPPR}) that the Pythagorean equation $x^{2}+y^{2}=z^{2}$ is PPR in $x,y$, equivalently, BPR in $x,y\mid z$. However, by \Cref{pr: Moreira triples} it is not Ramsey PR in $(x,y)\mid z$, hence it is also not Ramsey PR in $(x,y),z$.\footnote{This could have also been deduced directly from \Cref{cor: poly applications}.}
\end{enumerate}
\end{esem}

The following general result, close in spirit to the polynomial applications developed so far, shows that many equations involving sums of arbitrary functions are not Ramsey PR.

\begin{pr}\label{lemma mcpyth}
  Let $f\from \mathbb N\to \mathbb N$ be such that $n\mapsto f(n+1)-f(n)$ is finite-to-one. Then for every $g\from \mathbb N\to \mathbb  N$ the equation $g(x)+f(y)=f(z)$ is not Ramsey PR in $(x,y)\mid z$.
\end{pr}

\begin{proof} As always, by \Cref{teo: nonstandard first order properties RW} it suffices to fix $(\alpha,\beta)\models\mathrm{RW}$ and $\gamma\in\star{\N}$ and prove that $g(\alpha)+f(\beta)\neq f(\gamma)$.  By \Cref{co:RWbasics}(\ref{point:RWfintoone}) we have $f(\beta+1)-f(\beta)>g(\alpha)$, hence $f(\beta)<f(\beta)+g(\alpha)<f(\beta+1)$. It follows that $f(\beta)+g(\alpha)$ is not in the image of $f$.
\end{proof}

\begin{esem} By \Cref{lemma mcpyth} the equation $x+2^{y}=2^{z}$ is not Ramsey PR in $(x,y)\mid z$. \end{esem}

One may ask whether there is a version of \Cref{thm: no sums and products RW} involving products and exponentials in place of sums and products. This, and much more, is indeed true, and can be proven by using the following observation.

\begin{pr}\label{lemma:minedhbfs}
Let $\alpha,\beta$ be infinite and $f,g\from \mathbb N^2\to \mathbb N$. If there are standard functions $\phi_0, \phi_1,h$ such that
\begin{enumerate}
\item $\phi_0(f(\alpha,\beta))=\phi_1(\beta)$,
\item $\phi_0(g(\alpha,\beta))=h(\phi_1(\beta))$, and
\item\label{point:mined3} $h(\phi_1(\beta))\ne \phi_1(\beta)$,
\end{enumerate}
then $f(\alpha,\beta)\nsim g(\alpha,\beta)$.
\end{pr}

\begin{proof}
 Otherwise $\phi_0(f(\alpha,\beta))\sim \phi_0(g(\alpha,\beta))$, hence $  \phi_1(\beta)\sim h(\phi_1(\beta))$, and this would imply $h(\phi_1(\beta)) = \phi_1(\beta)$.
\end{proof}
Note that assumption~\ref{point:mined3} is true if, e.g., $h$ has finitely many fixed points.

Let us see some example applications of \Cref{lemma:minedhbfs} to some equations involving exponentiation. We begin with the promised ``products and exponentials'' version of \Cref{thm: no sums and products RW}. 

\begin{eg}\label{eg:expprod}
  The formula $(x^y=z) \land (xy=t)$ is not Ramsey PR in $(x,y)\mid z,t$.
\end{eg}

\begin{proof}
    Let $f(x,y)=x^y$ and $g(x,y)=xy$, and suppose there is $(\alpha,\beta)\models \mathrm{RW}$ with $f(\alpha,\beta)\sim g(\alpha,\beta)$. We apply \Cref{lemma:minedhbfs} by showing that these two object have different, but very close, exponential level (see e.g.~\cite[Definition~5.2]{berarducciSurrealNumbersDerivations2018}).

Let $\ell\coloneqq\ell_2$, and denote by $\ell^{(k)}$ the $k$-fold composition of $\ell$ with itself.
An easy computation shows that $\ell^{(3)}(f(\alpha,\beta))$ is at finite distance from $\ell^{(2)}(\beta)$, while $\ell^{(3)}(g(\alpha,\beta))$ is at finite distance from $\ell^{(3)}(\beta)$. From this observation we can conclude by applying \Cref{lemma:minedhbfs} with $\phi_0\coloneqq \ell^{(3)}$, $\phi_1\coloneqq\ell^{(2)}$ and $h(x)\coloneqq\ell(x)+n$ for a suitable finite $n\in \mathbb Z$.
  \end{proof}
  \begin{eg}\label{eg:yxxy}
  The formula $(x^y=z) \land (y^x=t)$ is not Ramsey PR in $(x,y)\mid z,t$. 
\end{eg}

\begin{proof}
  This is similar to the proof of \Cref{eg:expprod}, with the same $\phi_0$ and $\phi_1$.
\end{proof}
In particular, there is no pair $(\alpha,\beta)\models \mathrm{TS}$ such that $\alpha^\beta\sim \beta^\alpha$, solving a question from~\cite{luperibagliniExponentiationsUltrafilters2025}.
  \begin{eg}
The formula $(x2^y=z) \land (x+y=t)$ is not Ramsey PR in $(x,y)\mid z,t$.
    \end{eg}
    \begin{proof}
      If there were $(\alpha,\beta)\models\mathrm{RW}$ with $\alpha2^\beta\sim \alpha+\beta$, then $(\gamma,\delta)\coloneqq(2^\alpha,2^\beta)$ would be in $\mathrm{RW}$ and satisfy $\gamma^\delta\sim \gamma\delta$, against \Cref{eg:expprod}.
    \end{proof}

Although we deduced the previous example from \Cref{eg:expprod}, it is also a special case of the following.
    
    \begin{eg}\label{thm:hbfs}
      Let $f\from \mathbb N\to \mathbb N$ be an arbitrary function, $m>1$ a natural number, and $g(x,y)\from \mathbb N^2\to \mathbb N$ a function as follows. There is $d$ such that, for every $n$, we have $g(n,y)\in \mathcal O(y^d)\cap \Omega(y^{\frac 1d})$, that is, there are functions $h_0,h_1\from  \mathbb N\to \mathbb N$ such that for all $n,y$ we have $h_0(n)y^\frac 1d\le g(n,y)\le h_1(n)y^d$.  If $(\alpha,\beta)\models\mathrm{RW}$, then $f(\alpha) \cdot m^\beta\nsim g(\alpha,\beta)$.
    \end{eg}
    \begin{proof}
      This too is similar to the proof of \Cref{eg:expprod}, using $\phi_0=\ell_m^{(2)}$, $\phi_1=\ell_m$, and $h$ of the form $\ell_m(x)+n$.
    \end{proof}

We conclude this section with a result on polynomial functions. We state it directly in terms of Ramsey's witnesses.
    \begin{pr}\label{pr:polypr}
      Let $n,r\in \mathbb Z$, $m,s\in \mathbb N$, and $(\alpha,\beta)\models \mathrm{RW}$.
      \begin{enumerate}
      \item\label{point:nmrslinear} If $n,r,n+m,r+s\ne 0$ and  $n\alpha+m\beta\sim r\alpha+s\beta$, then $(n,m)=(r,s)$.
              \item Let\label{point:polyinbeta} $f(x,y)=\sum_{i=0}^m f_i(x)y^i$ and $g(x,y)=\sum_{j=0}^s g_j(x)y^j$, where  $f_i, g_j\from \mathbb N\to  \mathbb Z$ and $f_m(\alpha)\ne 0\ne g_s(\alpha)$. If $f(\alpha,\beta)\sim g(\alpha,\beta)$, then $m=s$.
      \item If $n,r>0$ and  $\alpha^n\beta^m\sim \alpha^r\beta^s$ then $(n,m)=(r,s)$.
      \end{enumerate}
    \end{pr}
    \begin{proof}
      \begin{enumerate}[wide=0pt]
      \item By \Cref{co:RWbasics}(\ref{point:RW8}) we have $n\alpha+m\beta \asymp \beta\asymp r\alpha+s\beta$. We apply \Cref{tecnico lemma 2}(\ref{point:teclem1}) and conclude that $m=s$. Now let $p$ be a finite prime much larger than $n+m+r$. Let $c\coloneqq\smod_p(\alpha)=\smod_p(\beta)$. If $v_p(\alpha)<v_p(\beta)$, then $\smod_p(n\alpha+m\beta)=nc$ and $\smod_p(r\alpha+m\beta)=rc$, and it follows that $n=r$. If $v_p(\alpha)=v_p(\beta)$ then, by \Cref{rem:exsecondacifra},
        \[
(n+m)c \equiv         \smod_p(n\alpha+m\beta)\equiv  \smod_p(r\alpha+m\beta)\equiv (r+m)c \pmod p,
\]
hence $n=r$.

\item Observe that $\ell_2(f(\alpha,\beta))\sim \ell_2(g(\alpha,\beta))$. There is $\epsilon\in \set{0,1}$ such that  $\ell_2(f(\alpha,\beta))=m\ell_2(\beta)+\ell_2(f_m(\alpha))+\epsilon$, hence $\st(f(\alpha,\beta)/m\ell_2(\beta))=1$. An analogous computation for $g(\alpha,\beta)$ shows that $\st(g(\alpha,\beta)/s\ell_2(\beta))=1$. It follows from \Cref{tecnico lemma 2}(\ref{point:teclem1}) applied to $f(\alpha,\beta)$ and $g(\alpha,\beta)$ that $\st(f(\alpha,\beta)/g(\alpha,\beta))=1$, hence $m=s$.

  \item We know from the previous point that $m=s$. Let $\pi(n)$ be the smallest prime dividing $n$. Observe that  $\pi(\alpha^a\beta^b)=\pi(\alpha)$. Write
    \begin{equation}
      \label{eq:2}      v_{\pi(\alpha)}(\alpha^a\beta^b)=av_{\pi(\alpha)}(\alpha)+bv_{\pi(\alpha)}(\beta).
    \end{equation}

 If $\pi(\alpha)<\pi(\beta)$, then the last summand is $0$. It follows that $nv_{\pi(\alpha)}(\alpha)\sim rv_{\pi(\alpha)}(\alpha)$, and since $v_{\pi(\alpha)}(\alpha)\ne 0$ we conclude by \Cref{pr:bddto1} that $n=r$.

If $\pi(\alpha)=\pi(\beta)$ we have two subcases. If $v_{\pi(\alpha)}(\alpha)<v_{\pi(\beta)}(\beta)=v_{\pi(\alpha)}(\beta)$  then this pair is in $\mathrm{RW}$ by \Cref{co:RWbasics}(\ref{point:RWfrw}). We then conclude by applying point~\ref{point:nmrslinear} above to this pair, using~\eqref{eq:2}. If instead $v_{\pi(\alpha)}(\alpha)=v_{\pi(\beta)}(\beta)$, then as $v_{\pi(\alpha)}(\beta)=v_{\pi(\beta)}(\beta)$, by~\eqref{eq:2} we have $(n+m)v_{\pi(\alpha)}(\alpha)=(r+m)v_{\pi(\alpha)}(\alpha)$, hence again $n=r$.\qedhere
      \end{enumerate}
    \end{proof}

\section{Open problems}\label{Section: open problems}

The results obtained in this paper open some lines of enquiry  that we believe worth investigating further.

Puritz' Theorem (\Cref{puritz}) gives a simple characterisation of tensor pairs in terms of standard functions. As for Ramsey pairs, in \Cref{co:RWbasics}, we have shown that they have several properties that can be expressed in terms of standard functions.  Most of our applications actually used some of these properties, so finding a Puritz-like characterisation of Ramsey pairs might be helpful in studying other similar combinatorial problems.

\begin{problem} \label{problem:charteriseRW}   Characterise Ramsey pairs in the style of Puritz' Theorem.\end{problem}

 \Cref{problem:charteriseRW} and nonstandard characterisations of related Ramsey properties are the topic of a paper in preparation. Another interesting question regards polynomial equations.

\begin{problem} Is it possible to extend \Cref{cor: poly applications} from monomials to arbitrary polynomials in $x$? To arbitrary multivariate polynomials?\end{problem}

Our proof of \Cref{cor: poly applications} uses the fact that it is simple to write the valuation of a monomial in $\alpha$ in terms of the valuation of $\alpha$; the same does not hold for a generic polynomial in $\alpha$. An extension of \Cref{cor: poly applications} to arbitrary polynomials in $x$ would give an analogous extension of its corollaries, and in particular of \Cref{prop Maleki}.

As observed in \Cref{rem:remainingcases}, some variants of the \nameref{thm:pwsp} remain open. The most interesting ones seem to be the following. We suspect that solving these should provide enough information to complete the classification.

\begin{problem}
  Which of these formulas  are Ramsey PR in $(x,y)\mid z,t$?
  \begin{enumerate}
  \item $(-x+y=z)\land (xy=t)$.
  \item $(-x+y=z)\land (y=tx)$.
  \item $(x+y=z)\land (y=tx)$.
  \end{enumerate}
\end{problem}

Our final question essentially asks to deal with the ``never considered'' cases of \Cref{tab:BlockPartitionRegularityOfVarphiXYZT,tab:RamseyPartitionRegularityOfVarphiXYZT,tab:RamseyPartitionRegularityOfVarphiXYZT2} from \Cref{sec:casestudy}.

\begin{problem} In which partitions of the variables is the formula $\left(x\neq y\right)\wedge \left(x+y=z\right)\wedge \left(x\cdot y= t\right)$ BPR? In which ones is it Ramsey PR?
\end{problem}

\footnotesize
\providecommand{\noopsort}[1]{}

\end{document}